\documentclass[11pt,a4paper]{article}
\usepackage[utf8]{inputenc}
\usepackage[english]{babel}
\usepackage{amssymb,amsfonts,amsmath,amsthm}

\usepackage{physics}
\usepackage{graphicx,color}
\usepackage[left=2cm,right=2cm,top=2cm,bottom=2cm]{geometry}

\usepackage{cite,hyperref,upgreek,enumerate}
\usepackage[normalem]{ulem}

\numberwithin{equation}{section}

\def \R {\mathbb{R}}

 \def \m {\mathbf{m}} \def \mm
{\widetilde{\mathbf{m}}}  

\def \q {\mathbf{q}}  

\def \r {\mathbf{r}}

\def \dbar {\bar{\delta}}

\def \div  {\nabla\cdot}
\def \grad {\nabla}

\def \sgn {\operatorname{sgn}}

\def \d{\partial}

\def \froeb[#1][#2]{}

\def \->{\rightarrow}
\def \=>{\Rightarrow}
\def \asm{\simeq}

\newtheorem{theorem}{Theorem}

\newtheorem{lemma}{Lemma}

\title{\bf Skyrmion stacking in stray field-coupled ultrathin
  ferromagnetic multilayers}

\author{N. J. Dubicki\thanks{Department of Mathematical Sciences,
    New Jersey Institute of Technology, Newark, New Jersey 07102,
    USA.} \and V. V. Slastikov\thanks{School of Mathematics,
    University of Bristol, Bristol BS8 1UG, United Kingdom.} \and
  A. Bernand-Mantel\thanks{CEMES, Université de Toulouse, CNRS, 29
    Rue Jeanne Marvig, BP94347, 31055 Toulouse, France.}  \and
  C. B. Muratov\thanks{Dipartimento di Matematica, Universit\`a di
    Pisa, Largo B. Pontecorvo, 5, 56127 Pisa, Italy.} \thanks{Corresponding author: {\tt
      cyrill.muratov@unipi.it}.}}

\begin{document}

\begin{titlepage}

\maketitle

\noindent\rule{7in}{1.2pt}\\

\begin{abstract}
  This paper explores the energy landscape of ferromagnetic multilayer
  heterostructures that feature magnetic skyrmions -- tiny whirls of
  spins with non-trivial topology -- in each magnetic layer. Such
  magnetic heterostructures have been recently pursued as possible
  hosts of room temperature stable magnetic skyrmions suitable for the
  next generation of low power information technologies and
  unconventional computing. The presence of stacked skyrmions in the
  adjacent layers gives rise to a strongly coupled nonlinear system,
  whereby the induced magnetic field plays a crucial stabilizing
  role. Starting with the micromagnetic modeling framework, we derive
  a general reduced energy functional for a fixed number of ultrathin
  ferromagnetic layers with perpendicular magnetocrystalline
  anisotropy. We next investigate this energy functional in the regime
  in which the energy is dominated by the intralayer exchange
  interaction and formally obtain a finite-dimensional description
  governed by the energy of a system of one skyrmion per layer as a
  function of the position, radius and the rotation angle of each of
  theses skyrmions. For the latter, we prove that energy minimizers
  exist for all fixed skyrmion locations. We then focus on the
  simplest case of stray field-coupled ferromagnetic bilayers and
  completely characterize the energy minimizers. We show that the
  global energy minimizers exist and consist of two stray
  field-stabilized N\'eel skyrmions with antiparallel in-plane
  magnetization components. We also calculate the energy of two
  skyrmions of equal radius as a function of their separation
  distance.
\end{abstract}
  
\noindent\rule{7in}{1.2pt}\\

\tableofcontents

\end{titlepage}

\section{Introduction}

Magnetic heterostructures have recently garnered significant interest
as promising platforms for hosting room-temperature stable magnetic
skyrmions -- tiny whirls of spins with non-trivial topology -- making
them ideal candidates for next-generation low-power information
technologies and unconventional computing
\cite{fert17,finocchio21,luo21}.  Magnetic skyrmions are particle-like
topologically protected magnetic textures predicted to exist in
ferromagnetic materials since the late 1980s
\cite{bogdanov89,bogdanov89a,bogdanov94}. They are typically
stabilized by chiral Dzyaloshinskii-Moriya interaction (DMI)
\cite{dzyaloshinskii58,moriya60} and have been observed in various
magnetic materials, including chiral magnets, ferromagnetic ultrathin
films and multilayers
\cite{heinze11,muhlbauer09,yu10,soumyanarayanan17,woo16}. Magnetic
skyrmions have been extensively investigated as potential information
carriers for spintronic devices \cite{nagaosa13,fert13,jiang17}.  Such
investigations are impossible without modeling and computational
studies of these systems. Modeling of magnetic skyrmions often relies
on two-dimensional models, a choice which may be justified in the case
of bulk chiral materials \cite{bogdanov89}, or for ultrathin magnetic
films, when the film thickness is much smaller than the exchange
length \cite{leonov16,bms:prb20}.
 
In the last 10 years, there has been a growing body of skyrmion
observations in multilayer heterostructures, whose intrinsic high
tunability enables the observation of skyrmions at room
temperature\cite{woo16,moreau-luchaire16}. These heterostructures may
present a total thickness that can be larger than the exchange length,
and, most importantly, the reduction of the interlayer exchange
interaction due to the use of non-magnetic spacers facilitates the
emergence of a thickness-dependent magnetization, leading to fully
three-dimensional magnetization textures, as observed
experimentally\cite{dovzhenko18,montoya17,legrand18}. Such twisting of
the magnetization in the direction perpendicular to the layers is a
well known phenomena related to the influence of ``magnetic charges''
in ferromagnetic films with out-of-plane magnetocrystalline
anisotropy. Indeed, the magnetic field induced by magnetic charges
modifies the internal structure of domain walls, a phenomenon widely
studied in the framework of bubble materials in the 1970's
\cite{schlomann73,hubert75}.

A similar phenomenon occurs in the case of multilayers with in-plane
magnetocrystalline anisotropy in the absence of exchange coupling,
i.e. in the case of purely stray field-coupled layers
\cite{garcia05,garcia06}.  For magnetic bilayers sufficiently thin to
be in the Néel wall regime, the lower-energy state consists of two
Néel walls with opposite rotations in both layers \cite{garcia05},
partially cancelling the stray field.  Similarly, in the case of
magnetic bilayers with perpendicular magnetic anisotropy it has been
observed experimentally that the Bloch wall, which is energetically
favorable in a single layer, is replaced by two Néel walls with the
opposite rotation senses \cite{bellec10}. This configuration can be
further stabilized by the DMI, assuming a proper choice of asymmetric
interfaces, as reported in the case of skyrmionic bubbles in bilayers
with opposite DMI constants \cite{hrabec17}. For larger number of
repeats, the competition between the dipolar and the DMI interactions
leads to twisted walls which present an asymmetry in the thickness
direction with respect to the middle of the magnetic multilayer. This
twist has been evidenced experimentally in the case of domain
walls\cite{fallon18,lucassen19} and skyrmionic
bubbles\cite{legrand18}. The asymmetric twisted walls have also been
the subject of detailed analytical modeling using domain wall and
skyrmionic bubble \emph{ans\"atze}, where the impact of twist on the
spin-orbit torque dynamics of skyrmionic bubbles has been studied
\cite{legrand18a,lemesh18}.

In the present work, we focus on a study of compact magnetic skyrmions
in ultrathin ferromagnetic multilayers, i.e., magnetic
heterostructures whose total thickness lies below the Bloch wall width
of the ferromagnetic material.  We use the tools of asymptotic
analysis to investigate multilayer systems with one skyrmion per layer
in the case of purely stray field-coupled layers. We start with the
micromagnetic energy functional for a three-dimensional multilayer
system that includes the intralayer exchange, the magnetocrystalline
anisotropy (both of bulk and interfacial origin), the Zeeman, the
interfacial DMI and the full stray field interactions.  We then derive
a reduced energy functional in the case where each ferromagnetic layer
is thin compared to the exchange length and the whole stack is thin
compared to the characteristic length scale of variation of the
magnetization in the film plane. We show that in addition to the usual
local shape anisotropy term and the non-local dipolar interaction
terms that are already present in the case of single ultrathin
ferromagnetic layers \cite{kmn:arma19,ms:prsa17,
  bms:prb20,dms:mmmas24}, the expansion of the stray field energy for
multilayers leads to the appearance of a new local dipolar energy term
corresponding to interlayer volume-surface interactions. This term
cancels out in the case of identical magnetizations in the adjacent
layers, but becomes equivalent to a stabilizing layer-dependent
interfacial DMI when the in-plane magnetization components in two
consecutive layers are opposite to one another.

The obtained reduced energy functional is investigated in the regime
in which the energy is dominated by the intralayer exchange
interaction, which corresponds to the conformal limit studied in
\cite{bms:arma21} in the case of single ultrathin ferromagnetic
layers. Following the arguments of \cite{bms:arma21}, we formally
obtain a finite-dimensional description of the system of interacting
compact magnetic skyrmions with one skyrmion per layer as a function
of the position, radius and the rotation angle of each of theses
skyrmions that is expected to be asymptotically exact in the
considered limit. We then investigate the energy landscape of the
above system in the case of zero applied magnetic field and prove that
energy minimizers exist for all fixed skyrmion positions. The
difficulty in obtaining such a result is that a priori it is not clear
whether the energy could not be reduced by some skyrmions shrinking to
zero radius. We prove that this phenomenon does not occur, if the
skyrmion centers are fixed.

We then apply our results to the case of stray field-coupled
ferromagnetic bilayers in the absence of DMI, where we obtain a
complete characterization of global energy minimizers in the
intralayer exchange-dominated regime. These minimizers consist of two
concentric Néel skyrmions with anti-parallel in-plane magnetization
and a prescribed chirality.  We also provide an expression for the
energy of two skyrmions of equal radius as a function of their
separation distance and illustrate our findings with a result of
micromagnetic simulations. Notice that extending such a
characterization to the general case of multilayers presents a
difficulty that the energy minimizing sequences could consist of
skyrmions in different layers moving far apart. For bilayers, we show
that this cannot occur by an explicit analysis of all the interaction
terms, which, however, becomes intractable for higher numbers of
layers.

Our paper is organized as follows. In section \ref{s:model}, we
specify the full three-dimensional micromagnetic energy functional
with all the relevant energy terms in the multilayer geometry and
carry out its non-dimensionalization. In section \ref{s:reduced}, we
explicitly compute the energy of the magnetizations that are
independent of the thickness variable in each ferromagnetic layer. We
then carry out an asymptotic expansion of the energy as the
ferromagnetic layer and interlayer thicknesses go to zero, with the
number of layers fixed, to obtain a reduced two-dimensional
variational model of ultrathin ferromagnetic multilayers with
perpendicular magnetic anisotropy. Then, in section \ref{s:Nskyr} we
introduce an ansatz in the form of one truncated Belavin-Polyakov
skyrmion in each ferromagnetic layer, which is asymptotically valid in
the intralayer exchange-dominated regime and asymptotically compute
the finite-dimensional energy function of such a skyrmion stack. In
section \ref{s:landscape}, we then explore the basic properties of the
obtained energy function and prove that it admits a global energy
minimizer for all fixed locations of the skyrmion centers, see Theorem
\ref{t:FN}. Finally, in section \ref{s:bilayer} we completely
characterize the global energy minimizer in the simplest non-trivial
case of ultrathin stray field-coupled bilayers in the intralayer
exchange-dominated regime, see Theorem \ref{t:bilayer}. We also
corroborate our conclusions with the results of micromagnetic
simulations and characterize the interaction energy of two skyrmions
of equal radius separated by a prescribed distance.

\paragraph{Acknowledgments.} N.J.D. and C.B.M. were partially
supported by NSF via grant DMS-1908709.  A. Bernand-Mantel was
supported by France 2030 government investment plan managed by the
French National Research Agency under grant reference PEPR SPIN
[SPINTHEORY] ANR-22-EXSP-0009 and grant NanoX ANR-17-EURE-0009 in the
framework of the Programme des Investissements d'Avenir. C.B.M. is a
member of INdAM-GNAMPA and acknowledges partial support by the MUR
Excellence Department Project awarded to the Department of
Mathematics, University of Pisa, CUP I57G22000700001, and by the PRIN
2022 PNRR Project P2022WJW9H. All analytical calculations in the paper
involving special functions were carried out, using {\sc Mathematica
  14.2} software.

\section{Micromagnetic model}
\label{s:model}

Our starting point is the micromagnetic modeling framework (in the SI
units), in which the observed magnetization configurations are
interpreted as local or global energy minimizers of a micromagnetic
energy functional evaluated on vector fields
$\mathbf M : \overline{\Omega} \to \R^3$ of fixed length
$|\mathbf M| = M_s$ \cite{landau8,hubert,brown}. Here
$\Omega \subseteq \R^3$ is the spatial domain (an open set) occupied
by a single centrosymmetric crystalline ferromagnetic material,
$\mathbf M = \mathbf M(\mathbf r)$ is the magnetization vector at
point $\mathbf r = (x, y, z) \in \Omega$, and $M_s > 0$ is the
saturation magnetization (in A/m). The energy functional (in J) that
contains the exchange, bulk magnetocrystalline anisotropy, Zeeman,
stray field, interfacial magnetocrystalline anisotropy and interfacial
DMI contributions, in that order, reads
\begin{align}
  \label{eq:E3dM}
  \mathcal E(\mathbf M)
  & = \int_\Omega \left\{ {A \over M_s^2} |\nabla
    \mathbf M|^2 + K_u \Phi_u \left ( {\mathbf M \over M_s} \right) -
    \mu_0 \mathbf H_a \cdot \mathbf M - 
    \frac{\mu_0}{2} \mathbf H_d \cdot \mathbf M \right\} d^3 r \notag
  \\
  & + \int_{\partial \Omega} \left\{
    K_s \Phi_s \left( {\mathbf M \over M_s}, \mathbf r \right) +
    {D_s\over M_s^2} \left( M^\| 
    \nabla_\perp \cdot \mathbf 
    M^\perp - \mathbf M^\perp \cdot \nabla_\perp M^\| \right) 
    \right\} d \mathcal H^2(\r).  
\end{align}
Here $A > 0$ is the exchange stiffness (in J/m), $K_u \geq 0$ is the
bulk magnetocrystalline anisotropy constant (in J/m$^3$),
$\Phi_u: \mathbb S^2 \to \R^+= [0, \infty)$ specifies the dependence
of the bulk anisotropy energy on the direction of $\mathbf M$,
$\mu_0 = 4 \pi \times 10^{-7}$ H/m is the vacuum permeability,
$\mathbf H_a \in \R^3$ is the applied magnetic field in (A/m), and
$\mathbf H_d : \R^3 \to \R^3$ is the demagnetizing field (in A/m)
produced by $\mathbf M$ via the solution of the stationary Maxwell's
equations
\begin{align}
  \label{eq:maxwell}
  \nabla \cdot (\mathbf H_d + \mathbf M) = 0, \qquad \nabla \times
  \mathbf H_d = 0,
\end{align}
distributionally in $\R^3$, with $\mathbf M$ extended by zero in
$\R^3 \backslash \overline{\Omega}$ \cite{dmrs:sima20}. The
interfacial terms in the second line of \eqref{eq:E3dM} contain the
interfacial DMI constant $D_s : \partial \Omega \to \R$ (in J/m),
which may take different values depending on the adjacent non-magnetic
material, the DMI energy term written in terms of
$\mathbf M = (\mathbf M^\perp, M^\|)$, where $M^\|$ is the component
of $\mathbf M$ along the normal to $\partial \Omega$ and
$\mathbf M^\perp$ is the tangential component to $\partial \Omega$,
respectively, and $\nabla_\perp$ is the tangential gradient,
$K_s : \partial \Omega \to \R^+$ is the interfacial magnetocrystalline
anisotropy constant (in J/m$^2$), which may also take different values
depending on the adjacent non-magnetic material, and
$\Phi_s : \mathbb S^2 \times \partial \Omega \to \R^+$ specifies the
dependence of the interfacial anisotropy energy on the magnetization
orientation relative to the crystalline axes in $\Omega$ and the
normal to $\partial \Omega$.  When $\Omega$ is unbounded, one often
needs to subtract the contribution of a fixed reference configuration
$\mathbf M^\star$ from the integrand in \eqref{eq:E3dM} to make the
resulting integrals convergent.

As usual, we introduce the normalized magnetization vector
$\mathbf m : \Omega \to \mathbb S^2$, extended by zero to the rest of
$\R^3$. In terms of $\mathbf m$ the energy becomes
\begin{align}
  \label{eq:E3dm}
  \mathcal E(M_s \mathbf m)
  &
    = \int_\Omega \left( A |\nabla \mathbf
    m|^2 + K_u \Phi_u(\mathbf m) - 2 K_d \mathbf h \cdot \m - K_d
    \mathbf 
    h_d \cdot \mathbf m \right) 
    d^3 r \notag \\
  & + \int_{\partial \Omega} \left(K_s \Phi_s(\m, \r) + D_s (m^\|
    \nabla_\perp \cdot \m^\perp - 
    \m^\perp \cdot \nabla m^\|)  \right) d \mathcal H^2(\r),
\end{align}
in which $K_d = \frac12 \mu_0 M_s^2$, $\mathbf h = \mathbf H_a / M_s$,
$\mathbf h_d = \mathbf H_d / M_s$, $\m = (\m^\perp, m^\|)$ on
$\partial \Omega$ as before. Note that the demagnetizing field
$\mathbf h_d$ satisfies \cite{dmrs:sima20}
\begin{align}
  \label{eq:hdU}
  \mathbf h_d = -\nabla U, \qquad \Delta U = \nabla \cdot \mathbf m,   
\end{align}
distributionally in $\R^3$. In particular, the magnetostatic potential
$U$ is given by
\begin{align}
  \label{eq:U}
  U(\mathbf r) = - \int_{\mathbb R^3} {\nabla \cdot \mathbf m(\mathbf r')
  \over 4 \pi |\mathbf r - \mathbf r'|} \, d^3 r', \qquad \mathbf r \in
  \R^3.
\end{align}
The quantity $\uprho_\m = -\nabla \cdot \mathbf m$ hence has the
meaning of the magnetic charge density distributionally in $\R^3$,
which includes the regular contribution of the volume charges from the
absolutely continuous part of $\uprho_\m$ in $\Omega$ and a singular
contribution to $\uprho_\m$ of the surface charges from the jump of
$\mathbf m$ to zero across $\partial \Omega$.

\begin{figure}
\centering
\includegraphics[width=10cm]{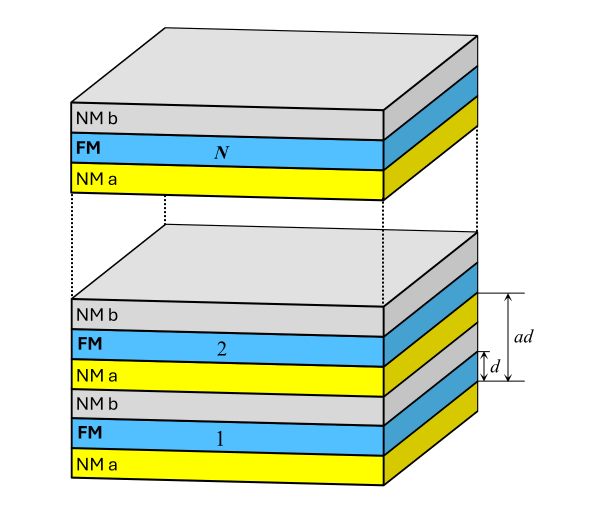}
\caption{Schematics of the geometry of a multilayer system. The
  heterostructure consists of $N$ repeats of a sandwich of thickness
  $ad$ in the form of a layer of one non-magnetic material (NM a),
  followed by a layer of a ferromagnet (FM) of thickness $d$, followed
  by a layer of another non-magnetic material (NM b), from the bottom
  to the top.}
\label{f:layers}
\end{figure}

We now specify the geometry of interest, which is that of a system of
$N$ identical ferromagnetic layers of thickness $d$ separated by
non-magnetic spacers of thickness $(a - 1) d$ with $a > 1$, so that the
total thickness of the magnetic layer plus the non-magnetic layer is
$ad$:
\begin{align}
  \label{eq:Om3d}
  \Omega = \R^2 \times \bigcup_{n=1}^N  \big[ (n-1)ad, (n-1)ad
  + d \big],
\end{align}
see Fig. \ref{f:layers}. The magnetic material is assumed to be
uniaxial, with an easy axis perpendicular to the layers, so that
$\Phi_u(\mathbf m) = |\mathbf m^\perp|^2$, where we now use the
convention
$\mathbf m(\mathbf r) = (\mathbf m^\perp(\mathbf r),
m^\parallel(\mathbf r))$ with $\mathbf m^\perp(\mathbf r) \in \R^2$
and $m^\parallel(\mathbf r) \in \R$ being the in-plane and
out-of-plane components of the magnetization, respectively, for all
$\mathbf r \in \R^3$. Similarly, we assume that the interfacial
anisotropy penalizes the tangential component of the
magnetization. Hence, we take $\Phi_s(\m, \r) = |\m^\perp|^2$ as
well. As is common in the spintronic multilayer materials
\cite{moreau-luchaire16,soumyanarayanan17}, the spacer is assumed to
generally consist of two sublayers, so the upper surface of each
magnetic layer is characterized by the interfacial DMI constant
$D_s^+$ and the interfacial anisotropy constant $K_s^+$, while the
bottom surface is characterized by the parameters $D_s^-$ and $K_s^-$,
respectively. Notice that in the case of a single material spacer we
simply set the parameters of the two surfaces equal to each
other. Finally, we assume that the applied field is perpendicular to
the film plane: $\mathbf h = h \hat{\mathbf z}$ for $h \in \R$.

As can be easily seen from the above choices, for $K > K_d$, where
\begin{align}
  \label{eq:K}
  K = K_u + {K_s^+ + K_s^- \over d}  
\end{align}
is the total magnetocrystalline anisotropy constant that takes into
account the interfacial magnetocrystalline anisotropy, the uniform
magnetization prefers to point out of the film plane. Indeed, if
$\mathbf m$ is constant across $\Omega$, the solution of the Maxwell's
equation in \eqref{eq:maxwell} is given by $\mathbf h_d = -\mathbf m$
in $\Omega$, resulting in the energy density
$(K - K_d) |\mathbf m^\perp|^2$ minimized by
$\mathbf m = \pm \hat{\mathbf z}$. Letting $\chi_n$ be the
characteristic function of the $n$-th interval
$I_n = \big( (n-1)ad, (n-1)ad + d \big) \subset \R$, we hence define
the reference configuration in which the magnetization points downward
in all the layers and is extended by zero to the whole of $\R^3$ as
\begin{align}
  \label{eq:mstar}
  \mathbf m^\star(x, y, z) = - \sum_{n=1}^N \hat{\mathbf z} \chi_n(z),   
\end{align}
whose associated demagnetizing field is
$\mathbf h_d^\star = -\nabla U^\star = -\mathbf m^\star$, where
$U^\star$ is a bounded solution of \eqref{eq:hdU} with
$\m = \m^\star$.  Subtracting its contribution from the integrand in
\eqref{eq:E3dm} and choosing the units of energy and length to be
given by $Ad$ and the exchange length $\ell_{ex} = \sqrt{A / K_d}$,
respectively, we can write the resulting renormalized energy
functional (non-dimensionalized) as
\begin{align}
  \label{eq:E}
  E(\mathbf m)
  & = {1 \over \delta} \sum_{n=1}^N
    \int_{(n-1)a\delta}^{((n-1)a + 1)\delta} \int_{\R^2}
    \left( |\nabla \mathbf
    m|^2 + Q_u |\mathbf m^\perp|^2 - 2 h  (m^\| +
    1) + \mathbf m \cdot \nabla U - 1 \right)
    d^2 r \, dz \notag \\
  &  + \sum_{n=1}^N \int_{\R^2} \left( Q_s^-
    |\m^\perp|^2 + 2 \kappa^- \m^\perp  
    \cdot \nabla m^\| \right) \bigg|_{\{z = (n-1)a \delta\}} \, d^2 r
  \\ 
  & + \sum_{n=1}^N \int_{\R^2} \left( Q_s^+ 
    |\m^\perp|^2 -2  \kappa^+ \m^\perp  
    \cdot \nabla m^\| \right) \bigg|_{\{z = ((n-1)a + 1)
    \delta\}} \, d^2 r, \notag
\end{align}
where we integrated the DMI by parts as in \cite{bms:prb20,bms:arma21}
to make it more compact and mathematically well-behaved. Here
\begin{align}
  \label{eq:dQ}
  \delta = {d \over \ell_{ex}}, \qquad Q_u = {K_u \over K_d}, \qquad
  Q_s^\pm = {K_s^\pm \over d K_d}, \qquad \kappa^\pm = {D_s^\pm \over
  d \sqrt{A K_d}}
\end{align}
are the dimensionless ferromagnetic layer thickness, the
magnetocrystalline bulk and interfacial anisotropies' quality factors,
and the dimensionless DMI strengths on the top and bottom surfaces of
the film, respectively. Note that with this definition
$E(\m^\star) = 0$.

\section{Reduced model for ultrathin multilayers}
\label{s:reduced}

We now consider the situation in which each ferromagnetic layer is
thin compared to the exchange length, i.e., the case $\delta \ll 1$,
when the variations of the magnetization along the $z$-direction in
each layer are highly penalized. In this case, following
\cite{garcia99}, it is appropriate to consider the magnetizations
$\mathbf m = \mm$ that do not vary across each layer and hence have
the particular form
\begin{equation}
  \mm(x,y,z) = \sum_{n=1}^{N} \m_n(x,y)\chi_n(z) = \sum_{n=1}^{N} \left(
    \m_n^\perp (x,y) + \hat{\mathbf z} m_n^\| (x, y) \right) \chi_n(z) ,
\end{equation}
where $\m_n : \R^2 \to \mathbb S^2$ for each $n = 1, \ldots, N$,
$\m_n^\perp$ is treated as an in-plane vector in $\R^3$, and now
\begin{align}
  \label{eq:chin}
  \chi_n(z) =
  \begin{cases}
    1 & \text{if} \ z \in [(n-1)a\delta, ((n-1)a + 1)\delta], \\
    0 & \text{otherwise.}
  \end{cases}
\end{align}

Focusing on the renormalized stray field energy, which for
configurations $\m_n(x, y)$ converging to $-\hat{\mathbf z}$
sufficiently fast as $x^2 + y^2 \to \infty$ can be rewritten, after
integrations by parts, as
\begin{align}
  \label{eq:Ed}
  E_d(\mm) = \frac{1}{\delta}\int_{\R^3} \left( |\grad U|^2 - \sum_n
             \chi_n \right)d^3r.
\end{align}
We define $\mm_r = \mm - \m^\star$ and $U_r= U - U^\star$, and observe
that
\begin{equation}
  E_d(\mm) = \frac{1}{\delta}\int_{\R^3} \left( |\grad U_r|^2 - 2
    \sum_{n=1}^N \chi_n \d_z
    U_r \right) d^3r = E_d' (\mm)+ E_d'' (\mm).
\end{equation}
Examining the second term, we note that 
\begin{align}
  \sum_{n=1}^N  \int_{\R^3}  \chi_n \d_z U_r \, d^3 r = - \int_{\R^3}
  \nabla U^\star \cdot \nabla U_r \, d^3 r =  \int_{\R^3}
  U^\star \, \nabla \cdot \mm_r \, d^3 r \notag \\
  = - \int_{\R^3}
  \mm_r \cdot \nabla U^\star \, d^3 r  = \sum_{n=1}^N \int_{\R^3}
  (m_n^\| + 1) \chi_n \, d^3 r,
\end{align}
after integrations by parts and using \eqref{eq:hdU}. Therefore, using
the fact that $|\m_n|=1$ and hence
$2(m_n^\| + 1) = |\m_n^\perp|^2 + |m_n^\|+1|^2$ for all $n$, we obtain
\begin{equation} \label{eq:Epp} E_d'' (\mm) = -\frac{2}{\delta}
  \sum_{n=1}^N \int_{\R^3} \chi_n \d_z U_r \, d^3r =
  -\sum_{n=1}^{N}\int_{\R^2} \left( |\m_{n}^{\perp}|^2 + |m_{n}^{\|} +
    1|^2 \right) d^2r.
\end{equation}

We now investigate
\begin{equation}
  E_d'(\mm) = \frac{1}{\delta}\int_{\R^3} |\grad U_r|^2 d^3r =
  \frac{1}{\delta}\int_{\R^3}\int_{\R^3}  \frac{\div\mm_r(\r )\div
    \mm_r(\r' )}{4\pi |\r - \r' |} \, d^3
 r \, d^3  r',
\end{equation}
as can be seen from \eqref{eq:U} and an integration by parts.
Explicitly, this integral reads
\begin{equation}
  E_d' (\mm) = \frac{1}{\delta}\int_{\R}\int_{\R}\int_{\R^2} \int_{\R^2}
  \frac{\div\mm_r(\r,z)\div \mm_r(\r',z')}{4\pi \sqrt{|\r-\r'|^2 +
      (z-z')^2}} \, d^2r\, d^2r' \, dz \, dz',
\end{equation}
where from now on the variables of integration $\r, \r' \in \R^2$.
Noting that
\begin{align}
  \partial_z \widetilde m_r^\|(\r, z)  = \sum_{n=1}^n( m_n^\|(\r)
  + 1) [\updelta_{a(n - 1) \delta}(z) - \updelta_{(a(n-1) + 1)\delta}(z) ],
  \qquad \forall \r \in \R^2,
\end{align}
where $\updelta_\alpha(z)$ is the Dirac delta centered at
$\alpha \in \R$, and integrating in $z$ and $z'$, we obtain
\begin{equation}
\begin{aligned}
  E'_d (\mm)= \frac{1}{\delta} \sum_{n=1}^{N} \sum_{k=1}^{N}
  \int_{\R^2}\int_{\R^2} &\bigg\{ K_{vv}^{a(k-n)}(|\r-\r'|) \
  \div\m_n^\perp(\r) \div\m_k^\perp(\r')\\
  &+ K_{vs}^{a(k-n)}(|\r-\r'|) \ \div\m_n^{\perp}(\r) (m_k^{\|}(\r')+1) \\
  &+ K_{sv}^{a(k-n)}(|\r-\r'|) \ (m_n^{\|}(\r) +1)  \div\m_k^\perp(\r') \\
  &+ K_{ss}^{a(k-n)}(|\r-\r'|) \ (m_n^{\|}(\r) +1) (m_k^{\|}(\r')+1)
  \bigg\} \, d^2r \, d^2r',
  \end{aligned}
\end{equation}
where for $u \in \mathbb R$ we defined the volume-volume,
volume-surface and surface-surface charge interaction kernels as
\begin{align}
  K_{vv}^u(r) &=  {1 \over 4 \pi} \int_0^\delta \int_{u
                \delta}^{(u+1)\delta} {dz' \, dz \over 
                \sqrt{r^2 + (z - z')^2}}, \\ 
  K_{vs}^u(r) &=  {1 \over 4 \pi} \int_0^\delta {dz \over
                \sqrt{r^2 + (z - u \delta)^2}} - {1 \over 4 \pi}
                \int_0^\delta {dz \over
                \sqrt{r^2 + (z - (1 + u) \delta)^2}},    \\ 
  K_{sv}^u(r) &= {1 \over 4 \pi} \int_0^\delta {dz' \over 
                \sqrt{r^2 + (z' + u \delta)^2}} - {1 \over 4 \pi}
                \int_0^\delta {dz' \over 
                \sqrt{r^2 + (z' - (1 - u) \delta)^2}},  \\ 
  K_{ss}^u(r) &=  {1 \over 2 \pi \sqrt{r^2 + 
                u^2 \delta^2}} - {1 \over 4 \pi \sqrt{r^2 + 
                (u - 1)^2 \delta^2}} - {1 \over 4 \pi \sqrt{r^2 + 
                (u + 1)^2 \delta^2}} . 
\end{align}
We see clearly that $K_{vs}^u(r) = K_{sv}^{-u}(r)$, allowing these
terms to be combined in the expression for $E_d'$.  For simplicity, we
express the energy as the sum of the interaction energies:
\begin{equation}
  \label{eq:Edp3}
  E'_d (\mm) = \sum_{n=1}^{N}  \sum_{k=1}^{N} \left\{ E_{vv}^{(nk)}
    (\mm) + E_{vs}^{(nk)}  (\mm) + E_{ss}^{(nk)}  (\mm) \right\},
\end{equation}
where
\begin{align}
  E_{vv}^{(nk)}  (\mm) &= \frac{1}{\delta} \int_{\R^2} \int_{\R^2}                         
                         K_{vv}^{a(k-n)} (|\r - \r'|) \div
                         \m_n^\perp(\r) \div \m_k^\perp(\r')
                         \, d^2 r \, d^2 r', \\ 
  E_{vs}^{(nk)}  (\mm)  &= - \frac{2}{\delta} \int_{\R^2} \int_{\R^2} 
                          K_{vs}^{a(k-n)} (|\r - \r'|)  
                          \m_n^\perp(\r)  \cdot \nabla m_k^\| (\r')
                          \, d^2 r \, d^2 r', \\ 
  E_{ss}^{(nk)} (\mm) &= \frac{1}{\delta} \int_{\R^2} \int_{\R^2}  
                        K_{ss}^{a(k-n)} (|\r - \r'|) ( m_n^\|(\r) + 1
                        ) (m_k^\| (\r') + 1) 
                        \, d^2 r \, d^2 r',
\end{align}
and we integrated by parts in the second line.

We can explicitly evaluate the interaction kernels:
\begin{align}
  K_{vv}^u(r) &=  \frac{1}{4\pi} \bigg[  2 \sqrt{r^2+
                u^2 \delta ^2}-\sqrt{r^2+ (1 + u)^2 \delta^2}-\sqrt{r^2+
                (1 - u)^2 \delta^2}   \notag \\
              & \quad + (1 + u) \delta \sinh ^{-1}\left(\frac{
                (1 + u) \delta}{r}\right) 
                + (1 - u) \delta \sinh
                ^{-1}\left(\frac{ (1 
                - u) \delta}{r}\right)  - 2
                u \delta  \sinh
                ^{-1}\left(\frac{
                u \delta}{r}\right) \bigg], \\ 
  K_{vs}^u(r) &= \frac{1}{4\pi} \left[ 2 \sinh ^{-1}\left(\frac{
                u \delta}{r}\right)-\sinh ^{-1}\left(\frac{
                (u+1) \delta}{r}\right)-\sinh ^{-1}\left(\frac{ (u - 1)\delta
                }{r}\right) \right] .  
\end{align}
The obtained expressions for $K_{vv}^s$, $K_{vs}^u$ and $K_{ss}^u$ may
be viewed as generalizations of those obtained in \cite{garcia99} for
$u = 0$. Notice that by an explicit calculation we have
\begin{align}
  \label{eq:intKss}
  \int_{\R^2} K_{ss}^u(|\r|) \, d^2 r =
  \begin{cases}
    \delta (1 - |u|), & |u| \leq 1, \\
    0, & |u| > 1.
  \end{cases}
\end{align}
By a similar explicit calculation
\begin{align}
  \label{eq:intKvs}
  \int_{\R^2} K_{vs}^u (|\r|) \, d^2 r = {\delta^2 \over 2} \, \text{sgn}(u), \qquad |u|
  > 1.
\end{align}
Also notice that $K_{vs}^0(r) = 0$, which can also be seen immediately
from the symmetry considerations.

We can further simplify the surface-surface interaction term. First we
observe that
\begin{equation}
\begin{aligned}
  \left( m_n^\|(\r) - m_n^\|(\r') \right)\left( m_k^\|(\r) -
    m_k^\|(\r') \right) &= \left(m_n^\|(\r) + 1
  \right)\left(m_k^\|(\r) + 1 \right) + \left(m_n^\|(\r') + 1
  \right)\left(m_k^\|(\r') + 1 \right) \\ 
  &- \left(m_n^\|(\r) + 1 \right) \left(m_k^\|(\r') + 1 \right) -
  \left(m_n^\|(\r') + 1 \right) \left(m_k^\|(\r) + 1 \right).
\end{aligned}
\end{equation}
Integrating with the kernel $K_{ss}^{a(k-n)}$, we find due to the
invariance of $K_{ss}^{a(k-n)}(|\r-\r'|)$ with respect to
interchanging $\r$ and $\r'$ that
\begin{equation}
\begin{aligned}
  \int_{\R^2}\int_{\R^2} & K_{ss}^{a(k-n)} (|\r-\r'|) \left(
    m_n^\|(\r) - m_n^\|(\r') \right)\left( m_k^\|(\r) - m_k^\|(\r')
  \right) \ d^2r
  \ d^2r'  \\
  &= 2 \int_{\R^2}\int_{\R^2} K_{ss}^{a(k-n)} (|\r-\r'|)
  \left(m_n^\|(\r) + 1
  \right)\left(m_k^\|(\r) + 1 \right) \ d^2r \ d^2r'   \\
  &- 2 \int_{\R^2}\int_{\R^2} K_{ss}^{a(k-n)} (|\r-\r'|)
  \left(m_n^\|(\r) + 1 \right) \left(m_k^\|(\r') + 1 \right) \ d^2r \
  d^2r' .
\end{aligned}
\end{equation}
Therefore, the surface-surface self-interaction energy for layer $n$
may be written as
\begin{equation}
  \label{eq:surfsurfnn}
  E_{ss}^{(nn)}(\mm) =
  \int_{\R^2} |m_n^\| + 1 |^2 d^2 r - \frac{1}{2\delta}
  \int_{\R^2}\int_{\R^2} K_{ss}^0(|\r-\r'|) \left( m_n^\|(\r) -
    m_n^\|(\r') \right)^2 \ d^2r \ d^2r',
\end{equation}
where we noted that by \eqref{eq:intKss} we have
$\int_{\R^2} K_{ss}^0(|\r|) d^2 r = \delta$.  At the same time, the
surface-surface interaction energy for different layers $n \not= k$ is
\begin{equation} \label{eq:surfsurfnk}
E_{ss}^{(nk)} (\mm) = - \frac{1}{2\delta} \int_{\R^2}\int_{\R^2}
  K_{ss}^{a(k-n)} (|\r-\r'|) \left( m_n^\|(\r) - m_n^\|(\r')
  \right)\left( m_k^\|(\r) - m_k^\|(\r') \right) \ d^2r \ d^2r',
\end{equation}
where again we used \eqref{eq:intKss} and noted that
$\int_{\R^2} K_{ss}^u(|\r|) d^2 r = 0$ for all $u > 1$, recalling that
$a > 1$ and $|k - n| \geq 1$. Then, combining all three interactions
in \eqref{eq:Edp3} with \eqref{eq:Epp} we get an exact expression for
the energy of $\mm$:
\begin{align}
  \label{eq:Emmexact}
  E_d(\mm) = -\sum_{n=1}^N \int_{\R^2} |\m^\perp_n|^2 \, d^2 r +
  \frac{1}{\delta} \sum_{n=1}^N \sum_{k=1}^N  
  \int_{\R^2} \int_{\R^2} K_{vv}^{a(k-n)} (|\r - \r'|) \div 
  \m_n^\perp(\r) \div \m_k^\perp(\r')
  \, d^2 r \, d^2 r' \notag \\ 
  - \frac{2}{\delta} \sum_{n=1}^N \sum_{k=1}^N  \int_{\R^2}
  \int_{\R^2}  
  K_{vs}^{a(k-n)} (|\r - \r'|)  
  \m_n^\perp(\r)  \cdot \nabla m_k^\| (\r')
  \, d^2 r \, d^2 r' \notag \\
  - \frac{1}{2\delta} \sum_{n=1}^N \sum_{k=1}^N
  \int_{\R^2}\int_{\R^2} 
  K_{ss}^{a(k-n)} (|\r-\r'|) \left( m_n^\|(\r) - m_n^\|(\r')
  \right)\left( m_k^\|(\r) - m_k^\|(\r') \right) \ d^2r \ d^2r'. 
\end{align}

We now simplify this cumbersome expression in such a way that it is
valid asymptotically as $\delta \to 0$ with $\mm$ and all the other
parameters fixed. Since the layer displacement parameter, $u$, is
understood to be fixed with respect to $\delta \to 0$, one expands to
find the following asymptotic behaviors for $\delta \ll 1$:
\begin{align}
  K_{vv}^u(r) &=  \frac{\delta^2}{4\pi r} +
                o(\delta^2), \label{eq:Kvvs} \\
  K_{vs}^u(r) &=  \frac{u\delta^3}{4\pi r^3} + o(\delta^3),  \label{eq:Kvs} \\
  K_{ss}^u(r) &= \frac{\delta^2}{4\pi r^3} + o(\delta^3).
\end{align}
Therefore, for $\m_n$ sufficiently smooth the expression in
\eqref{eq:surfsurfnn} may be rendered asymptotically for $\delta \ll 1$ as
\begin{equation}
  \label{eq:surfasmnn}
  E_{ss}^{(nn)} (\mm) =
  \int_{\R^2} | m_n^\| + 1 |^2 \, d^2 r - \delta
  \int_{\R^2}\int_{\R^2} \frac{ ( m_n^\|(\r) - m_n^\|(\r') )^2 }{ 8\pi
    |\r-\r'|^3 } \ d^2r \ d^2r' + o(\delta),
\end{equation}
where we noted that the strong singularity of the kernel is partially
cancelled by $(m_n^\|(\r) - m_n^\|(\r'))^2 = O(|\r - \r'|^2)$ for
$m_n^\| \in C^1(\R^2)$, making the integral in the right-hand side of
\eqref{eq:surfasmnn} convergent. Similarly, we have
\begin{equation}
  E_{ss}^{(nk)} (\mm) \asm - \delta \int_{\R^2}\int_{\R^2} \frac{ (
    m_n^\|(\r) - m_n^\|(\r') )( m_k^\|(\r) - m_k^\|(\r') ) }{ 8\pi
    |\r-\r'|^3 } \ d^2r \ d^2r' + o(\delta) 
\end{equation}
for all $n \not= k$. Meanwhile, using \eqref{eq:Kvvs} the
volume-volume interactions can be asymptotically expressed as
\begin{equation}
\begin{aligned}
  E_{vv}^{(nk)} (\mm) = \delta \int_{\R^2} \int_{\R^2} \frac{
    \div\m_n^\perp(\r) 
    \div\m_k^\perp(\r') }{4\pi |\r-\r'| } \ d^2r \ d^2r'
  +o(\delta).
\end{aligned}
\end{equation}

As can be seen from \eqref{eq:Kvs}, we have
$\delta^{-2} K_{vs}^u(r) \to 0$ as $\delta \to 0$ for all $r >
0$. Care, however, is needed in passing to the limit in the integral,
as the strong singularity of the kernel in \eqref{eq:Kvs} precludes
passing the pointwise limit to the value of $E_{vs}^{(nk)}(\mm)$ at
$O(\delta)$. In fact, from \eqref{eq:intKvs} and the above observation
it is clear that for all $|u| > 1$ we have
\begin{align}
  {2 \over \delta^2} K_{vs}^u(|\r|) \to \text{sgn}(u) \updelta (\r)
  \qquad \text{as } \delta \to 0,
\end{align}
in the sense of distributions, where $\updelta(\r)$ is a Dirac delta
centered at the origin in $\R^2$.  As a consequence, by an argument
similar to the one in the case of the surface-surface interactions the
volume-surface interaction energies admit the following asymptotic
expansion:
\begin{equation}
  E_{vs}^{(nk)}  (\mm) = -\text{sgn}(k-n) \delta \int_{\R^2} 
  \m^\perp_n \cdot \nabla m_k^\| d^2 r + o(\delta).
\end{equation}

Thus, combining all the terms as in \eqref{eq:Emmexact}, we obtain
that to within $o(\delta)$ accuracy for $\delta \ll 1$ and with all
other parameters fixed, the total energy is asymptotically
$E(\mm) \asm E_N(\{\m_n\})$, where
\begin{align}
  \label{eq:E0}
  E_N(\{\m_n\})
  &
    = \sum_{n=1}^N \int_{\R^2} \left( |\nabla \m_n|^2 + (Q -
    1) |\m_n^\perp|^2 - 2 h (m_n^\| + 1) - 2 \kappa \m_n^\perp
    \cdot \nabla m_n^\| \right) d^2 r \notag \\
  & - \delta \sum_{n=1}^{N-1} \sum_{k=n+1}^N \int_{\R^2}
    \left( 
    \m^\perp_n \cdot \nabla m_k^\|  - 
    \m^\perp_k  \cdot \nabla m_n^\| \right) \, d^2 r \\
  & + \delta \sum_{n=1}^{N} \sum_{k=1}^{N} \int_{\R^2}\int_{\R^2} \bigg(
    \frac{ \div\m_n^\perp(\r) \div\m_k^\perp(\r') }{4\pi|\r-\r'|} -
    \frac{ (m_n^{\|}(\r) - m_n^\| (\r') )( m_k^{\|}(\r)- m_k^\|(\r')
    ) }{8\pi|\r-\r'|^3} \bigg) d^2r\ d^2r' , \notag
\end{align}
where we defined the material quality factor and the dimensionless DMI
strength 
\begin{align}
  \label{eq:Qkappa}
  Q = {K \over K_d}, \qquad \kappa = \kappa^+ - \kappa^-,
\end{align}
respectively. We note that, as can be seen from the above derivation,
the asymptotic formula in \eqref{eq:E0} is rigorously valid in the
limit $\delta \to 0$ at least for every $\m_n : \R^2 \to \mathbb S^2$
such that $\m_n + \hat{\mathbf z} \in C^\infty_c(\R^2 ; \R^3)$ for
every $n = 1, \ldots, N$, and by an approximation argument also
extends to the natural class of functions
$\m_n + \hat{\mathbf z} \in H^1(\R^2; \R^3)$.

The obtained energy functional in \eqref{eq:E0} generalizes the one
obtained in \cite{bms:prb20, bfbsm:prb23} for a single layer to the
case of multilayers. Notice that in addition to the non-local dipolar
interactions present in the case of a single-layer, a peculiar new
local term appears in the second line of \eqref{eq:E0} that
corresponds to the interaction of the out-of-plane magnetic moments
with the approximately vertical magnetic field formed by the volume
charges immediately above and below the layers (see also
\cite{lemesh18}). This term vanishes, as expected, in the case when
all layers have identical magnetizations, in which case the energy is
equivalent to that of the magnetization of a single layer of thickness
$N \delta$. However, when the in-plane magnetizations in two layers
are opposite to one another: $\m^\perp_n = -\m^\perp_k$, this term
works as an effective interfacial DMI term, favoring N\'eel rotation
in these layers.

The reduced energy is applicable in the situation in which the
characteristic scale of variation of $\m_n(x, y)$ exceeds the
thickness $a N \delta$ of the entire stack. Note that in this regime
the energy is independent of $a$ and, hence, does not see the presence
of the non-magnetic spacers. This property is known to be violated
once the magnetization configurations acquire the lateral size
comparable to the stack thickness \cite{legrand18}, in which case the
full kernels $K^u_{vv}$, $K^u_{vs}$ and $K^u_{ss}$ need to be
utilized, resulting in a much more cumbersome model that only permits
a numerical treatment \cite{legrand18a,lemesh18,buttner18}.


\section{Energy of $N$ stacked skyrmions}
\label{s:Nskyr}

We have demonstrated that for ultrathin multilayers the micromagnetic
energy functional of a magnetization configuration
$\m : \overline{\Omega} \to \mathbb S^2$ obeys
$E(\m) \simeq E_N(\{\m_n\})$, where $E_N(\{\m_n\})$ is given by
\eqref{eq:E0}. This energy may be further simplified by taking
advantage of the scaling properties of different terms in the
energy. For $Q > 1$, introducing a rescaling and the new parameters:
\begin{align}
  \label{eq:scaledb}
  \r \to {\r \over \sqrt{Q - 1}}, \qquad \bar \delta = {\delta \over
  \sqrt{Q - 1}}, \qquad \bar h = {h \over Q - 1}, \qquad \bar \kappa =
  {\kappa \over \sqrt{Q - 1}}, 
\end{align}
we see that
$E_N(\{ \m_n(\cdot / \sqrt{Q - 1}) \}) = \bar E_N(\{ \m_n \})$, where
\begin{align}
  \label{eq:EN}
  \bar E_N(\{\m_n
  & \}) =
    \sum_{n=1}^N \int_{\R^2} \left( |\nabla \m_n|^2 +
    |\m_n^\perp|^2 - 2 \bar h (m_n^\| + 1) - 2 \bar \kappa
    \m_n^\perp 
    \cdot \nabla m_n^\| \right) d^2 r \notag \\
  & - \bar\delta \sum_{n=1}^{N-1} \sum_{k=n+1}^N \int_{\R^2}
    \left( \m^\perp_n  \cdot \nabla m_k^\| - 
    \m^\perp_k  \cdot \nabla m_n^\|\right) \, d^2 r \\
  & + \bar\delta \sum_{n=1}^{N} \sum_{k=1}^{N} \int_{\R^2}\int_{\R^2} \bigg(
    \frac{ \div\m_n^\perp(\r) \div\m_k^\perp(\r') }{4\pi|\r-\r'|} -
    \frac{ (m_{n}^{\|}(\r) - m_n^\| (\r') )( m_k^{\|}(\r)- m_k^\|(\r')
    ) }{8\pi|\r-\r'|^3} \bigg) d^2r\ d^2r' . \notag
\end{align}
Hence in the following we focus our attention on the study of the
energy $\bar E_N$ for the magnetization configurations consisting of a
single skyrmion in each ferromagnetic layer. The latter may be
specified by prescribing the topological degree +1 to the
magnetization in each layer:
\begin{align}
  \label{eq:degreen}
  \mathcal N(\m_n) = {1 \over 4 \pi} \int_{\R^2} \m_n \cdot
  (\partial_x \m_n \times \partial_y \m_n) \, d^2 r = 1 \qquad \forall
  n = 1, \ldots, N,
\end{align}
which together with some additional technical assumptions should
ensure existence of local minimizers of the energy
\cite{bms:prb20,bms:arma21}.

For $N = 1$, it is known that for $\bar h = 0$ the energy $\bar E_N$
admits minimizers in an appropriate function class for all
$0 < \bar \delta < \bar \delta_0$, with $\bar \delta_0 > 0$ universal
\cite{bms:prb20}. Moreover, the energy-minimizing profiles admit a
complete asymptotic characterization in the conformal limit
$\bar \kappa, \bar \delta \to 0$ \cite{bms:arma21}. After a suitable
translation, dilation and rotation, these profiles approach the
canonical Belavin-Polyakov (BP) profile
\begin{align}
  \m_{\infty}(\r) = \left( -\frac{2 \r}{1 + |\r|^2}, 
  \frac{1-|\r|^2}{1+|\r|^2} \right).
\end{align}
This is due to the fact that for $\bar\delta \ll 1$ the energy of a
skyrmion in a single layer gets close to the value of the Dirichlet
energy of harmonic maps with degree +1, whose minimizers are
well-known \cite{belavin75,eells88,bms:arma21}. Furthermore, the
rigidity estimates for almost harmonic maps of degree +1 and a fine
analysis of the tail of the skyrmion profiles yield closeness of the
profiles to those of the form
$\m(\r) = R \m_\infty( (\r - \r_0)/\rho)$, where $R \in SO(3)$ is a
rotation around the $z$-axis \cite{bms:arma21}. As the anisotropy
energy evaluated on $\m_\infty$ diverges, in order to determine the
skyrmion radius $\rho$ and the rotation angle $\theta$ as functions of
$\bar \kappa$ and $\bar \delta$ one needs to consider a truncated BP
profile, whose energy yields the asymptotic dependence of the skyrmion
characteristics on $\bar \kappa, \bar \delta \to 0$ \cite{bms:arma21}.

To that end, we define
\begin{equation}
f(r) = \frac{2r}{1 + r^2}, \qquad r \geq 0,
\end{equation}
and its truncated version
\begin{equation}
f_L(r) = 
\begin{cases}
  f(r) \ , & \text{if } r\leq \sqrt{L}, \\
  \frac{f(\sqrt{L})}{K_1(1/\sqrt{L})} K_1(r/L) \ , & \text{if } r >
  \sqrt{L},
\end{cases}
\end{equation}
where $K_1(x)$ is the modified Bessel function of the second
kind. This choice of the truncation is motivated by the asymptotic
decay of the skyrmion solution at infinity to the leading order in
$\bar \kappa, \bar \delta \ll 1$ determined by the exchange and
anisotropy terms \cite{ivanov90}.  For $L > 1$ we then define
\begin{align}
 \label{eq:mL}
  \m_{\rho, \theta, L, \r_0} (\r) = \left( \ -f_L \left( {
  |\r-\r_0| \over \rho} \right) \frac{R_\theta (\r-\r_0) }{|\r-\r_0|} , \
  \sgn(\rho-|\r-\r_0|)\sqrt{1-f_L^2 \left( {|\r-\r_0| \over \rho} \right)
  } \
  \right),
\end{align}
where $R_\theta \in SO(2)$ is a counter-clockwise rotation by angle
$\theta \in [-\pi, \pi)$.  Using the above expression as an ansatz,
the task of determining the values of $\rho$ and $\theta$ for a
skyrmion amounts to minimizing the energy of
$\m_{\rho, \theta, L, \r_0} $ in $\rho$, $\theta$ and $L$ to the
leading order in $\bar \kappa, \bar \delta \ll 1$.

It is clear that the above considerations should remain valid for
several layers, each containing a single skyrmion. Therefore, we now
consider an asymptotic expansion of the energy $\bar E_N$ for the profiles
of the form (see Fig. \ref{f:ansatz} for an illustration):
\begin{align}
  \label{eq:prof}
  \m_n = \m_{\rho_n, \theta_n, L_n, \r_n}, \qquad n = 1, \ldots, N,
\end{align}
and study the energy landscape of
$\bar E_N(\{ \m_{\rho_n, \theta_n, L_n, \r_n}\})$ in terms of its
dependence on $\{\rho_n, \theta_n, L_n, \r_n\}$ for
$\bar \delta, |\bar\kappa| \ll 1$.  
Notice that despite having reduced
our problem to a finite-dimensional one, we still need to study a
strongly nonlinear, fully coupled system whose analysis is a
significant challenge. In particular, it is not a priori clear whether
the minimum of $\bar E_N(\{ \m_{\rho_n, \theta_n, L_n, \r_n}\})$ is
attained, as it may be energetically favorable for a skyrmion in one
of the layers to collapse, which would correspond to $\rho_n \to 0$,
violating \eqref{eq:degreen} in the limit.

\begin{figure}
  \centering
  \includegraphics[width=10cm]{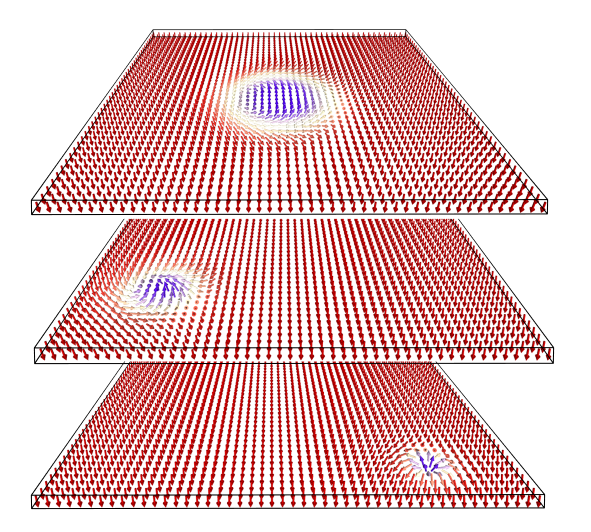}
  \caption{An example of a skyrmion configuration from \eqref{eq:prof}
    with three distinct radii and centers in a ferromagnetic
    trilayer. The skyrmion in the bottom layer is of N\'eel type
    ($\theta_1 = 0$), the skyrmion in the top layer is of Bloch type
    ($\theta_3 = \pi/2$), and the skyrmion in the middle layer is of
    mixed type ($\theta_2 = \pi/4$).}
  \label{f:ansatz}
\end{figure}

For simplicity of notation, we can write
\begin{align}
  \label{eq:ENsum}
  \bar E_N(\{ \m_{\rho_n, \theta_n, L_n, \r_n}\}) = \sum_{n=1}^N \left( 
  \bar E_\mathrm{ex}(\m_n) + \bar E_\mathrm{an}(\m_n) + \bar E_\mathrm{Z}(\m_n) +
  \bar E_\mathrm {DMI}(\m_n) \right)
  \notag \\
  + \sum_{n=1}^{N-1} \sum_{k=n+1}^N \bar E_{vs}(\m_n, \m_k) + \sum_{n=1}^N
  \sum_{k=1}^N \left( \bar E_{vv}(\m_n, \m_k) + \bar E_{ss}(\m_n, \m_k) \right) , 
\end{align}
where each term in the sum corresponds to the respective one in
\eqref{eq:EN}. As was shown in \cite{bms:arma21}, for all
$L_n \geq L_0$, with some $L_0 > 1$ universal, one can carry out an
expansion of each term in \eqref{eq:ENsum} that becomes asymptotically
exact as $L_0 \to \infty$. Utilizing the results from \cite[Lemma
A.6]{bms:arma21} and \cite{bfbsm:prb23}, we obtain that for
$L_0 \gg 1$ we have
\begin{align}
  \bar E_\mathrm{ex} ( \m_{n,\rho_n,\theta_n, L_n } ) - 8\pi
  & \asm 
    \frac{4\pi}{L_n^2}, \\
  \bar E_\mathrm{an}( \m_{n,\rho_n,\theta_n, L_n } )
  & \asm 4\pi \rho_n^2
    \ln\left( \frac{4L_n^2}{e^{2(1+\gamma)}} \right), \\
  \bar E_\mathrm{Z} ( \m_{n,\rho_n,\theta_n, L_n } )
  & \asm -4 \pi \bar h
    \rho_n^2 \ln \left( \frac{4L_n^2}{e^{1+2\gamma}} \right),  \\
  \bar E_\mathrm{DMI} ( \m_{n,\rho_n,\theta_n, L_n } )
  & \asm -8 \pi \bar \kappa
    \rho_n \cos \theta_n.
\end{align}

We next introduce the Fourier transform of
$\m_{r,n} = \m_n + \hat{\mathbf z} \in H^1(\R^2; \R^3)$:
\begin{align}
  \label{eq:mnhat}
  \widehat{\m}_{r,n}(\q) = \int_{\R^2} e^{-i \q \cdot \r} \m_{r,n}(\r) \, d^2 r
\end{align}
and write the stray field-mediated terms as \cite{lieb-loss}
\begin{align}
  \bar E_{vv}(\m_n,\m_k)
  &= \frac{\dbar}{2} \int_{\R^2} {(\q \cdot \overline{\widehat \m_{r,n}^\perp} ) ( \q
    \cdot \widehat \m_{r,k}^\perp )
    \over |\q|} {d^2 q \over (2 \pi)^2} , \\ 
  \bar E_{ss}(\m_n,\m_k)
  &= - \frac{\dbar}{2}  \int_{\mathbb{R}^{2}}
    |\q| \overline{\left( \widehat m^\|_{r,n}\right) } \widehat m^\|_{r,k}
    {d^2 q \over (2 \pi)^2} , \\ 
  \bar E_{vs}(\m_n,\m_k)
  &=- i \bar\delta \int_{\R^2}
    \q \cdot \left( \overline{\widehat \m_{r,n}^\perp}   \widehat m^\|_{r,k} - 
    \overline{\widehat \m_{r,k}^\perp}   \widehat m^\|_{r,n} \right)
    \, {d^2 q \over (2 \pi)^2} , 
\end{align}
where the overline denotes complex conjugate. 

We now assume that the magnetization $\m_n$ takes the form
\eqref{eq:prof}.  Using the fact that
\begin{equation}
  R_{\theta_n} \frac{\r-\r_n}{|\r-\r_n|} = 
  \frac{\r-\r_n}{|\r-\r_n|} \cos \theta_n+ 
  \frac{(\r-\r_n)}{|\r-\r_n|}^\perp \sin \theta_n, 
\end{equation}
where $(\r-\r_n)^\perp$ denotes the 90$^\circ$ counter-clockwise
rotation of $\r-\r_n$, we can split the in-plane component of $\m_n$
into two terms and note that the second term has zero divergence and
will not factor into the volume charge energy. Using the results from
\cite[Lemmas A.5 and A.6]{bms:arma21}, we obtain asymptotically for
$L_0 \gg 1$:
\begin{align}
  \bar E_{vv}(\m_n,\m_k)
  &\asm 2 \dbar \rho_n^2\rho_k^2 \cos\theta_n
    \cos\theta_k \int_{\R^2}  e^{i\q\cdot
    (\r_n-\r_k)} |\q| K_1(\rho_n|\q|)K_1(\rho_k|\q|)
    \, d^2 q,  \\
  \bar E_{ss}(\m_n,\m_k)
  &\asm -2 \dbar \rho_n^2\rho_k^2 \int_{\R^2}  e^{i\q\cdot
    (\r_n-\r_k)} |\q| K_0(\rho_n|\q|)K_0(\rho_k|\q|)  \
    d^2 q, \\ 
  \bar E_{vs}(\m_n,\m_k)
  &\asm -4 \dbar \rho_n^2\rho_k^2 \cos\theta_n\int_{\R^2}
    e^{i\q\cdot (\r_n-\r_k)} |\q| K_1(\rho_n|\q|)K_0(\rho_k|\q|)  \
    d^2 q \notag \\ 
  & \quad + 4  \dbar \rho_n^2\rho_k^2 \cos\theta_k  \int_{\R^2}
    e^{i\q\cdot (\r_k-\r_n)} |\q| K_0(\rho_n|\q|)K_1(\rho_k|\q|)  \
    d^2 q.
\end{align}
Integrating in polar coordinates with $q = |\q|$ and $\varphi$ being
the polar angle between the vectors $\r_{n} -\r_k$ and $\q$, we can
further reduce the above integrals with the help of the well-known
formula
\begin{equation}
  {1 \over 2\pi} \int_0^{2\pi} e^{ i q  |\r_{n} -\r_k| \cos \varphi}
  d\varphi =  
  J_0\left( q  |\r_{n} -\r_k| \right),
\end{equation}
where $J_0(x)$ is the Bessel function of the first kind.  Introducing
the new variables
\begin{equation}
  \alpha = \sqrt{\frac{\rho_k}{\rho_n}} \ , \ \ \beta =
  \sqrt{\rho_n\rho_k} \ , \ \ \lambda = \frac{ |\r_n
    -\r_k|}{\sqrt{\rho_n\rho_k}},
\end{equation}
and the rescaled variable of integration
$\xi = q \sqrt{\rho_n\rho_k}$, we can write the energies above as
\begin{align}
  \bar E_{vv}(\m_n,\m_k)
  &\asm 4 \pi \dbar \beta \cos\theta_n \cos\theta_k
    \int_0^\infty J_0(\lambda \xi) K_1(\alpha
    \xi)K_1(\xi/\alpha)  \ \xi^2 d\xi, \\ 
  \bar E_{ss}(\m_n,\m_k)
  &\asm -4\pi \dbar \beta \int_0^\infty  J_0(\lambda \xi) K_0(\alpha
    \xi)K_0(\xi/\alpha)  \ \xi^2 d\xi,\\ 
  \bar E_{vs}(\m_n,\m_k)
  &\asm -8\pi \dbar \beta \cos\theta_n \int_0^\infty  J_0(\lambda \xi)
    K_0(\alpha \xi)K_1(\xi/\alpha)  \ \xi^2 d\xi \notag \\ 
  &\quad + 8\pi \dbar \beta \cos\theta_k \int_0^\infty  J_0(\lambda
    \xi) K_1(\alpha \xi)K_0(\xi/\alpha)  \ \xi^2 d\xi. 
\end{align}

Now define the functions
\begin{align}
  F_{vv}(\alpha,\lambda)
  &= \frac{32}{3\pi^2} \int_0^\infty \xi^2 J_0(\lambda \xi) K_1(\alpha
    \xi) K_1(\xi/\alpha) d\xi, \label{eq:Fvvdef} \\  
  F_{ss}(\alpha,\lambda)
  &= \frac{32}{\pi^2} \int_0^\infty  \xi^2 J_0(\lambda \xi) K_0(\alpha
    \xi) K_0(\xi/\alpha) d\xi, \label{eq:Fssdef}\\ 
  F_{vs}(\alpha,\lambda)
  &= 2 \int_0^\infty  \xi^2 J_0(\lambda \xi)
    K_0(\alpha \xi) K_1(\xi/\alpha) d\xi, \label{eq:Fvsdef}
\end{align}
normalized so that $F_{vv}(1,0)=F_{ss}(1,0) =F_{vs}(1,0)= 1$. Notice
that all these three functions are uniformly bounded. Indeed, since
$K_{0,1}(t) > 0$ for all $t > 0$ and $|J_0(t)| \leq 1$, we have
$|F_{vv}(\alpha, \lambda)| \leq F_{vv}(\alpha, 0)$,
$|F_{ss}(\alpha, \lambda)| \leq F_{ss}(\alpha, 0)$ and
$|F_{vs}(\alpha, \lambda)| \leq F_{vs}(\alpha, 0)$, with the equality
achieved only for $\lambda = 0$ (see Lemmas \ref{l:Fvvss} and
\ref{l:Fvs} in the Appendix). At the same time, for $0 < \alpha < 1$
there holds
\begin{align}
  F_{vv}(\alpha, 0) & = \frac{16 \alpha \left(\left(\alpha^4+1\right)
                      E\left(1-\alpha^4\right)-2 \alpha^4 
   K\left(1-\alpha^4\right)\right)}{3 \pi
                      \left(\alpha^4-1\right)^2}, \label{eq:Fvv0} \\
  F_{ss}(\alpha, 0) & = \frac{16 \alpha^3
                      \left(\left(\alpha^4+1\right)
                      K\left(1-\alpha^4\right)-2 
   E\left(1-\alpha^4\right)\right)}{\pi
                      \left(\alpha^4-1\right)^2}, \label{eq:Fss0}
\end{align}
where $K(m)$ and $E(m)$ are the complete elliptic integrals of the
first and second kind, respectively.\footnote{We use the convention
  $K(m) = \int_0^{\pi/2} {d \theta \over \sqrt{1 - m \sin^2 \theta}}$
  and $E(m) = \int_0^{\pi/2} \sqrt{1 - m \sin^2 \theta} \, d \theta$.}
For $\alpha > 1$ one can infer the values of the above functions via
identities $F_{vv}(\alpha, 0) = F_{vv}(\alpha^{-1}, 0)$ and
$F_{ss}(\alpha, 0) = F_{ss}(\alpha^{-1}, 0)$. We also can explicitly
compute for any $\alpha > 0$
\begin{align}
  F_{vs}(\alpha, 0) & = \frac{2 \alpha^3 \left(\alpha^4-4 \ln
                      \alpha
                      -1\right)}{\left(\alpha^4-1\right)^2}, \label{eq:Fvs0} 
\end{align}
extending the last expression by continuity to $F_{vs}(\alpha, 0) = 1$
at $\alpha = 1$.  The graphs of the above functions are presented in
Fig. \ref{f:Fvvssvs0}.

\begin{figure}
  \centering
  \includegraphics[width=14cm]{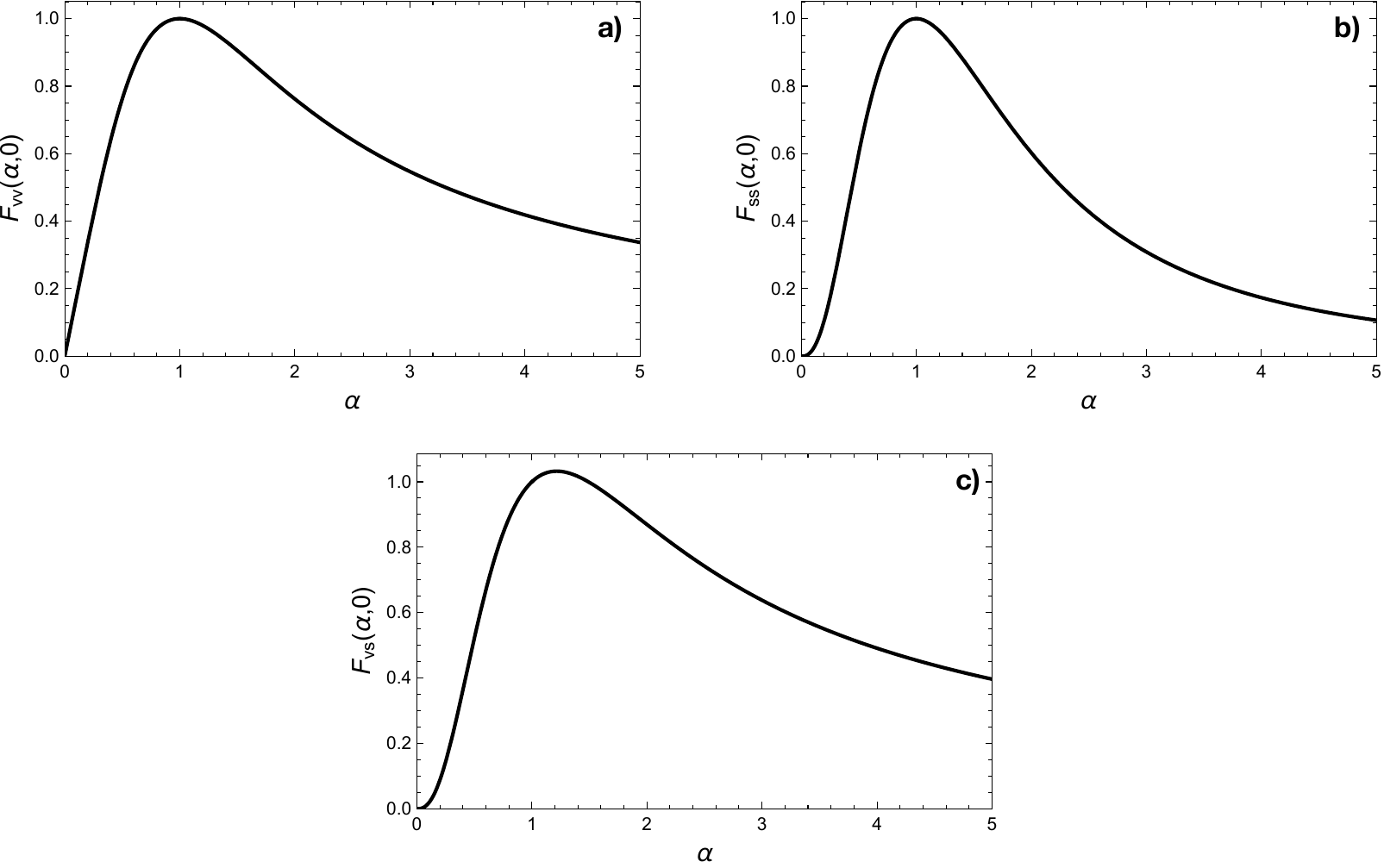}
  \caption{The plots of $F_{vv}(\alpha, 0)$ (a), $F_{ss}(\alpha, 0)$
    (b), and $F_{vs}(\alpha, 0)$ (c).}
  \label{f:Fvvssvs0}
\end{figure}

With these definitions, we have the following representation for the
energies:
\begin{align}
  \bar E_{vv}(\m_n,\m_k)
  &\asm \frac{3\pi^3}{8} \dbar \beta \cos\theta_n \cos\theta_k
    F_{vv}(\alpha,\lambda), \\ 
  \bar E_{ss}(\m_n,\m_k)
  &\asm -\frac{\pi^3}{8} \dbar \beta F_{ss}(\alpha,\lambda),\\ 
  \bar E_{vs}(\m_n,\m_k)
  &\asm -4\pi \dbar \beta \cos\theta_n F_{vs}(\alpha,\lambda) +  4\pi
    \dbar \beta \cos\theta_kF_{vs}(\alpha^{-1},\lambda).
\end{align} 
Coming back to the original set of parameters
$\{\rho_n, L_n,\theta_n,\r_n\}$, it is clear that
\begin{align}
  \bar E_{vv}(\m_{n},\m_{k})
  &\asm \frac{3\pi^3}{8} \dbar \sqrt{\rho_n\rho_k} \cos\theta_n
    \cos\theta_k F_{vv}\left( \sqrt{\frac{\rho_k}{\rho_n}}, \frac{
    |\r_{n} -\r_k|}{\sqrt{\rho_n\rho_k} } \right), \label{eq:Fvv} \\ 
  \bar E_{ss}(\m_{n},\m_{k}) 
  &\asm -\frac{\pi^3}{8} \dbar \sqrt{\rho_n\rho_k} F_{ss} \left(
    \sqrt{\frac{\rho_k}{\rho_n} },\frac{  |\r_{n} -\r_k| }{
    \sqrt{\rho_n\rho_k} } \right), \label{eq:Fss}\\ 
  \bar E_{vs}(\m_n,\m_k)
  &\asm -4\pi \dbar \sqrt{\rho_n\rho_k} \cos\theta_n
    F_{vs}\left(\sqrt{\frac{\rho_k}{\rho_n}},\frac{  |\r_{n} -\r_k| }{
    \sqrt{\rho_n\rho_k} }\right) \notag \\ 
  & \quad +  4\pi \dbar \sqrt{\rho_n\rho_k} \cos\theta_k
    F_{vs}\left(\sqrt{\frac{\rho_n}{\rho_k}},\frac{  |\r_{n} -\r_k| }{
    \sqrt{\rho_n\rho_k} }\right), \label{eq:Fvs}
\end{align} 
and from the definitions of $F_{vv}$ and $F_{ss}$ we observe that
\begin{align}
  \bar E_{vv} ( \m_{n}, \m_n )
  &\asm \frac{3\pi^3}{8} \dbar \rho_n
    \cos^2\theta_n ,\\  
  \bar E_{ss} ( \m_{n},\m_n)
  &\asm - \frac{\pi^3}{8} \dbar \rho_n.
\end{align}
Using the above expressions, we may asymptotically express the full
energy, to leading order in $L_0 \gg 1$ as
$\bar E_N(\{\rho_n, \theta_n, L_n, \r_n\}) - 8 \pi N \simeq
F_N(\{\rho_n, \theta_n, L_n, \r_n\})$, where
\begin{equation}
\begin{aligned}
  F_N(& \{\rho_n, \theta_n, L_n, \r_n\}) \\
  & = \sum_{n=1}^N \left[ \frac{4\pi}{L_n^2} + 4\pi \rho_n^2 \ln\left(
      \frac{4L_n^2}{e^{2(1+\gamma)}} \right) -4 \pi \bar h \rho_n^2
    \ln \left( \frac{4L_n^2}{e^{1+2\gamma}} \right) -8 \pi \bar \kappa
    \rho_n \cos \theta_n + \dbar
    \frac{\pi^3}{8} \rho_n (3\cos^2\theta_n  - 1) \right]  \\
  & + \sum_{n=1}^{N-1} \sum_{k=n+1}^N \dbar \frac{\pi^3}{4}
  \sqrt{\rho_n \rho_k} \left[ 3\cos\theta_n \cos\theta_k F_{vv} \left(
      \sqrt{\frac{\rho_k}{\rho_n} },\frac{ |\r_{n} -\r_k| }{
        \sqrt{\rho_n\rho_k} } \right) - F_{ss} \left(
      \sqrt{\frac{\rho_k}{\rho_n} },\frac{ |\r_{n} -\r_k| }{
        \sqrt{\rho_n\rho_k} } \right) \right] \\
  &-\sum_{n=1}^{N-1} \sum_{k=n+1}^N 4\pi \dbar \sqrt{\rho_n\rho_k}
  \left[ \cos\theta_n F_{vs}\left(\sqrt{\frac{\rho_k}{\rho_n}},\frac{
        |\r_{n} -\r_k| }{ \sqrt{\rho_n\rho_k} }\right) - \cos\theta_k
    F_{vs}\left(\sqrt{\frac{\rho_n}{\rho_k}},\frac{ |\r_{n} -\r_k| }{
        \sqrt{\rho_n\rho_k} }\right) \right].
\end{aligned}
\end{equation}
This is the reduced energy function that determines the energy
landscape of an $N$-skyrmion stack in terms of the skyrmion
parameters.



\section{The landscape of the energy function $F_N$}
\label{s:landscape}
We now investigate the energy landscape governed by the function $F_N$
in the regime of its applicability to stacked magnetic skyrmions in
ferromagnetic multilayers. To simplify the discussion, from now on we
assume that there is no applied external magnetic field, $\bar h = 0$,
and omit the Zeeman contribution to the energy from now on. The
analysis below can be easily extendable to the case $\bar h < 0$, when
the field is applied opposite to the magnetization direction in the
skyrmion core. At the same time, for fields applied along the
magnetization in the core, $\bar h > 0$, the energy landscape becomes
more complex due to skyrmion bursting \cite{bfbsm:prb23}, a new
phenomenon whose study goes beyond the present paper.

As we are interested in the minimization of $F_N$, we introduce
several auxiliary functions which are obtained by partial
minimization. With a slight abuse of notation, we still utilize the
symbol $F_N$ to denote those functions, which now depend on fewer
variables. For example, we introduce
\begin{equation}
\label{energy01}
\begin{aligned}
  F_N (\{\rho_n,\theta_n,\r_n\})=\min_{L_n > 0} F_N(\{\rho_n,
  \theta_n, L_n, \r_n\}),
\end{aligned}
\end{equation}
which is obtained by minimizing the energy function in all $L_n$. An
explicit calculation shows that $F_N(\{\rho_n, \theta_n, L_n, \r_n\})$
is minimized for $L_n = \rho_n^{-1}$, which is also the unique
critical point of $F_N(\{\rho_n, \theta_n, L_n, \r_n\})$ in $L_n$,
resulting in
\begin{equation}\label{energy02}
  \begin{aligned}
    F_N(& \{\rho_n, \theta_n,\r_n\}) \\
    & = \sum_{n=1}^N \left[ - 4\pi \rho_n^2 \ln\left(
        \frac{e^{1+2\gamma}}{4} \rho_n^2 \right) -8 \pi \bar \kappa
      \rho_n \cos \theta_n + \dbar
      \frac{\pi^3}{8} \rho_n (3\cos^2\theta_n  - 1) \right]  \\
    & + \sum_{n=1}^{N-1} \sum_{k=n+1}^N \dbar \frac{\pi^3}{4}
    \sqrt{\rho_n \rho_k} \left[ 3\cos\theta_n \cos\theta_k F_{vv}
      \left( \sqrt{\frac{\rho_k}{\rho_n} },\frac{ |\r_{n} -\r_k| }{
          \sqrt{\rho_n\rho_k} } \right) - F_{ss} \left(
        \sqrt{\frac{\rho_k}{\rho_n} },\frac{ |\r_{n} -\r_k| }{
          \sqrt{\rho_n\rho_k} } \right) \right] \\
    &-\sum_{n=1}^{N-1} \sum_{k=n+1}^N 4\pi \dbar \sqrt{\rho_n\rho_k}
    \left[ \cos\theta_n
      F_{vs}\left(\sqrt{\frac{\rho_k}{\rho_n}},\frac{ |\r_{n} -\r_k|
        }{ \sqrt{\rho_n\rho_k} }\right) - \cos\theta_k
      F_{vs}\left(\sqrt{\frac{\rho_n}{\rho_k}},\frac{ |\r_{n} -\r_k|
        }{ \sqrt{\rho_n\rho_k} }\right) \right].
\end{aligned}
\end{equation}
Note, however, that this value of $L_n$ is admissible only when
$L_n > L_0$, where $L_0 \gg 1$ is the parameter that measures the
applicability of the energy function $F_N$ to the system of $N$
stacked skyrmions (see section \ref{s:Nskyr}). This means that for
consistency we need to restrict the admissible values of $\rho_n$ to
$\rho_n < L_0^{-1}$ for a fixed $L_0 > 0$ sufficiently large. We thus
define, for $\{\r_n\}_{n=1}^N \subset \R^2$, an admissible class
\begin{align}
  \label{eq:ANrn}
  \mathcal A_N(\{\r_n\}) = \left\{ \rho_n \in ( 0, L_0^{-1}),  \ \theta_n
  \in [-\pi , \pi),  \ \r_n \right\}_{n=1}^N.
\end{align} 
Without loss of generality, from now on we may assume that
$L_0 > \bar L_0$, where
\begin{align}
  \label{eq:L0bar}
  \bar L_0 = \frac12 e^{2+\gamma},
\end{align}
which ensures that the first term in the first line of the right-hand
side of \eqref{energy02} is strictly convex in $\rho_n$.

We next investigate the energy $F_N(\{\r_n\})$ obtained by fixing the
positions of the skyrmions and minimizing with respect to angles and
radii. We note that as the admissible set of the radii
$0 < \rho_n < L_0^{-1}$ is not closed, the existence of solutions for
this minimization problem is not a priori clear, since some skyrmions
may prefer to collapse, $\rho_n \to 0$, or burst,
$\rho_n \to L_0^{-1}$, in the course of minimization.  As a first
step, we prove that the minimization of
$F_N(\{\rho_n,\theta_n,\r_n\})$ in $\theta_n$ and $\rho_n$ is well
defined, i.e., that minimizers of this problem exist. The obtained
minimal energy can be used to understand the interaction between
magnetic skyrmions at different locations. It also allows to estimate
from below the energy of the saddle points that separate the basins of
attraction of different skyrmion spatial arrangements. We will
subsequently illustrate the explicit solution of this problem in the
simplest case of stray field-coupled ferromagnetic bilayers with no
DMI.

\begin{theorem}
  \label{t:FN}
  Let $N \in \mathbb N$ and $L_0 > \bar L_0$. There exists
  $\bar\delta_0>0$ such that for any fixed $\{ \r_n \} \subset \R^2$
  and all $\dbar, |\bar \kappa| < \bar \delta_0$ there exists a
  minimizer of the problem
\begin{align}
  F_N(\{ \r_n \}) = \min_{\mathcal A_N(\{ \r_n \})}F_N(\{ \rho_n,
  \theta_n, \r_n\}). 
\end{align}
\end{theorem}
\begin{proof}
  \noindent {\bf Step 1.}  We first minimize in angles $\theta_n$,
  keeping $\{\rho_n,\r_n\}$ fixed, with $1\leq n \leq N$. It is clear
  that if we fix $\r_n \in \R^2$ and $\rho_n \in (0, L_0^{-1})$, the
  function $F_N(\{\rho_n, \theta_n, \r_n\})$ is continuous and
  periodic in $\theta_n$ with period $2\pi$ and, therefore, there
  exists a global minimizer $\{ \theta_n^* \} \subset
  [-\pi,\pi)^N$. We define the resulting minimal energy as
  $ F_N(\{\rho_n,\r_n\}) =\min_{\theta_n} F_N(\{ \rho_n, \theta_n,
  \r_n\})$, where
\begin{equation}\label{energy0222}
\begin{aligned}
  F_N(\{\rho_n,\r_n\}) &= \sum_{n=1}^N \left[ - 4\pi \rho_n^2
    \ln\left( \frac{e^{1+2\gamma}}{4} \rho_n^2 \right) -8 \pi \bar
    \kappa \rho_n \cos \theta_n^* + \dbar
    \frac{\pi^3}{8} \rho_n (3\cos^2\theta_n^*  - 1) \right]  \\
  &+ \sum_{n=1}^{N-1} \sum_{k=n+1}^N \dbar \frac{\pi^3}{4}
  \sqrt{\rho_n \rho_k} \left[ 3\cos\theta_n^* \cos\theta_k^* F_{vv}
    \left( \sqrt{\frac{\rho_k}{\rho_n} },\frac{ |\r_{n} -\r_k|}{
        \sqrt{\rho_n\rho_k} } \right) - F_{ss} \left(
      \sqrt{\frac{\rho_k}{\rho_n} },\frac{ |\r_{n} -\r_k| }{
        \sqrt{\rho_n\rho_k} } \right) \right] \\
  &-\sum_{n=1}^{N-1} \sum_{k=n+1}^N 4\pi \dbar \sqrt{\rho_n\rho_k}
  \left[ \cos\theta_n^*
    F_{vs}\left(\sqrt{\frac{\rho_k}{\rho_n}},\frac{ |\r_{n} -\r_k| }{
        \sqrt{\rho_n\rho_k} }\right) - \cos\theta_k^*
    F_{vs}\left(\sqrt{\frac{\rho_n}{\rho_k}},\frac{ |\r_{n} -\r_k| }{
        \sqrt{\rho_n\rho_k} }\right) \right].
\end{aligned}
\end{equation}

\noindent {\bf Step 2.} Now we show that $ F_N(\{\rho_n,\r_n\})$ is
bounded from below. From Lemmas \ref{l:Fvvss} and \ref{l:Fvs} in the
Appendix, we know that $F_{ss}(\alpha,\lambda)$,
$F_{vv}(\alpha,\lambda)$, and $F_{vs}(\alpha,\lambda)$ are bounded
functions. Then since $2\sqrt{\rho_n \rho_k} \leq \rho_n +\rho_k$, it
follows that for some $C>0$ universal we have
\begin{equation}\label{energy02221}
\begin{aligned}
  F_N(\{\rho_n,\r_n\}) &\geq \sum_{n=1}^N \left[ - 4\pi \rho_n^2
    \ln\left( \frac{e^{1+2\gamma}}{4} \rho_n^2 \right) - C(|\bar
    \kappa| + \dbar)
    \rho_n \right] \geq -C_1
\end{aligned}
\end{equation}
for some $C_1 > 0$ and all $0<\rho_n < L_0^{-1}$.

Therefore, for $1 \leq n \leq N$ there are minimizing sequences
$(\rho_{n,l})_l$ such that as $l \to \infty$
\begin{align}
 F_N(\{\rho_{n,l},\r_n\}) \to \inf_{\rho_n} F_N(\{\rho_n,\r_n\}) > -\infty.
\end{align}
Since $\rho_{n,l} \in \left(0, L_0^{-1} \right)$, up to extraction of
a subsequence (not relabeled) they converge:
\begin{align}
  \lim_{l \to \infty} \rho_{n,l} = \rho^*_n \in [0, L_0^{-1}].
\end{align}

\noindent {\bf Step 3.} We now show that in a minimizing sequence
$\rho_{n,l}$ do not converge to $0$. Assume this is not true and,
without loss of generality, there is a minimizing sequence with
$\rho_{N,l} \to 0$, while for $1 \leq n \leq N-1$ we have
$\rho_{n,l} \to \rho_n^* \geq 0$ as $l \to \infty$. Let us assume
$\rho_n^*>0$ for all $1\leq n \leq N-1$, as the other case is simpler
and will follow in a similar fashion. We also assume that
$\r_N \neq \r_n$ for all $1 \leq n \leq N-1$, as the other case is
also simpler. We now observe that the infimum of the energy is
achieved on $\{ \rho_n^* \}_{n=1}^{N-1} \cup \{0\}$:
\begin{equation}
\begin{aligned}
  \inf_{\rho_n} & F_N (\{\rho_n,\r_n\}) = \sum_{n=1}^{N-1} \left[ -
    4\pi (\rho_n^*)^2 \ln\left( \frac{e^{1+2\gamma}}{4} (\rho_n^*)^2
    \right) -8 \pi \bar \kappa \rho_n^* \cos \theta_n^* + \dbar
    \frac{\pi^3}{8} \rho_n^* (3\cos^2\theta_n^* - 1)
  \right]  \\
  &+ \sum_{n=1}^{N-2} \sum_{k=n+1}^{N-1} \dbar \frac{\pi^3}{4}
  \sqrt{\rho_n^* \rho_k^*} \left[ 3\cos\theta_n^* \cos\theta_k^*
    F_{vv} \left( \sqrt{\frac{\rho_k^*}{\rho_n^*} },\frac{ |\r_{n}
        -\r_k|}{ \sqrt{\rho_n^*\rho_k^*} } \right) - F_{ss} \left(
      \sqrt{\frac{\rho_k^*}{\rho_n^*} },\frac{ |\r_{n} -\r_k| }{
        \sqrt{\rho_n^*\rho_k^*} } \right) \right] \\
  &-\sum_{n=1}^{N-2} \sum_{k=n+1}^{N-1} 4\pi \dbar
  \sqrt{\rho_n^*\rho_k^*} \left[ \cos\theta_n^*
    F_{vs}\left(\sqrt{\frac{\rho_k^*}{\rho_n^*}},\frac{ |\r_{n} -\r_k|
      }{ \sqrt{\rho_n^*\rho_k^*} }\right) - \cos\theta_k^*
    F_{vs}\left(\sqrt{\frac{\rho_n^*}{\rho_k^*}},\frac{ |\r_{n} -\r_k|
      }{ \sqrt{\rho_n^*\rho_k^*} }\right) \right] \\
  & = F_{N-1}(\{ \rho_n^*, \r_n\}). \label{eq:rhostar}
\end{aligned}
\end{equation}
Therefore, we only need to show that there exists $\rho_N \neq 0$ such
that with $\rho_n = \rho_n^*$ for $1 \leq n \leq N - 1$ we have
\begin{align}
  F_N(\{\rho_n,\r_n\}) < F_{N-1}(\{\rho_n^*,\r_n\}).
\end{align}
Since $F_N(\{\rho_n,\r_n\})$ is obtained by minimizing in $\theta_n$,
from \eqref{eq:rhostar} we know that by taking $\theta_n=\theta_n^*$
for $1\leq n\leq N-1$ and $\theta_N= \frac{\pi}{2}$ we have
\begin{align}
  F_N(\{\rho_n, \r_n\})
  & \leq  F_N(\{\rho_n,  \theta_n, \r_n\}) 
    = F_{N-1}(\{\rho_n^*,\r_n\})
    -G(\{\rho_n, \r_n \}), 
\end{align}
where
\begin{align}
  G(\{\rho_n, \r_n \})
  &=4\pi \rho_N^2 \ln\left( \frac{e^{1+2\gamma}}{4} \rho_N^2
    \right)+\dbar \frac{\pi^3}{8} \rho_N  + \dbar \frac{\pi^3}{4}
    \sum_{n=1}^{N-1} \sqrt{\rho_n \rho_N}  F_{ss} \left(
    \sqrt{\frac{\rho_N}{\rho_n} },\frac{ |\r_N -\r_n| }{
    \sqrt{\rho_N\rho_n} } \right) \\ 
  & + 4\pi \dbar \sum_{n=1}^{N-1}   \sqrt{\rho_n\rho_N}
    \cos\theta_n^* F_{vs}\left(\sqrt{\frac{\rho_N}{\rho_n}},\frac{
    |\r_N -\r_n| }{ \sqrt{\rho_n\rho_N} }\right)  . 
\end{align}
We now need to show that $G(\{\rho_n, \r_n \})>0$ as then we will have
a contradiction with minimality. We will take $\rho_N>0$ small and
investigate $ F_{ss}(\alpha,\lambda) = F_{ss} (\alpha^{-1}, \lambda)$
defined in \eqref{eq:Fssdef}, where
$\alpha =\sqrt{\frac{\rho_N}{\rho_n} }$,
$\lambda=\frac{ |\r_N -\r_n| }{ \sqrt{\rho_N\rho_n} } $.  We can use
the change of variables $t= \xi \lambda$ to obtain
\begin{align}
  F_{ss}(\alpha,\lambda) = \frac{32}{\pi^2 \lambda^3}  \int_0^\infty
  t^2 J_0(t) K_0(\alpha t/\lambda) K_0(t/(\alpha \lambda)) dt.
\end{align}
Next we observe that for $\rho_N \ll 1$ the value of
$\beta={\alpha \over \lambda} =\frac{\rho_N}{|\r_N-\r_n|}>0$ is small,
while $\frac{1}{ \alpha \lambda} = \frac{\rho_n}{|\r_N-\r_n|}$ is
fixed. We have estimates on the Bessel function:
\begin{align}
0 < K_0(t) < h(t)=\max\{-\ln t +1, 1\}, \qquad \forall t > 0.
\end{align}
It is now clear that with the help of $|J_0(t)| \leq 1$ we can
estimate
\begin{align}
  \left|\int_0^\infty  t^2 J_0(t) K_0(\alpha t/\lambda) K_0(t/(\alpha
  \lambda)) dt \right| \leq \int_0^\infty  t^2 K_0( t/(\alpha
  \lambda)) h(\beta t) dt \notag \\
  = \ln \beta^{-1} \int_0^{\beta^{-1}} t^2 K_0(t/(\alpha \lambda)) dt
  + \int_0^{\beta^{-1}}  t^2 K_0( t/(\alpha
  \lambda)) |\ln t| dt  \\
  + \int_0^\infty  t^2 K_0( t/(\alpha
  \lambda)) dt \leq C|\ln \rho_N|, \notag
\end{align}
for some $C = C(\rho_n,|\r_N-\r_n|) > 0$ and all $\rho_N > 0$ small
enough. Since
$\lambda^{-3} = \left(\frac{\sqrt{\rho_N \rho_n}}{|\r_N-\r_n|}
\right)^3$, we obtain
\begin{align}
  \dbar \frac{\pi^3}{4}  \sum_{n=1}^{N-1} \sqrt{\rho_N \rho_n}  F_{ss}
  \left( \sqrt{\frac{\rho_n}{\rho_N} },\frac{ |\r_N -\r_n| }{
  \sqrt{\rho_N\rho_n} } \right) \geq -C \dbar \rho_N^2 |\ln \rho_N| 
\end{align}
for some $C = C(\rho_n,|\r_N-\r_n|) > 0$ and all $\rho_N>0$ small
enough.

We now use the same approach to estimate $F_{vs}$ defined in
\eqref{eq:Fvsdef}:
\begin{align}
  F_{vs}(\alpha,\lambda)
  &= 2 \int_0^\infty  k^2 J_0(\lambda k)
    K_0(\alpha k) K_1(k/\alpha) dk = \frac{2}{
    \lambda^3}  \int_0^\infty  t^2 J_0(t)
    K_0(\alpha t/\lambda) K_1(t/(\alpha
    \lambda)) dt. 
\end{align}
Using the fact that
\begin{align}
  |F_{vs}(\alpha,\lambda)| \leq \frac{2}{ \lambda^3}  \int_0^\infty
  t^2 K_1( t/(\alpha \lambda)) h(\beta t) dt \leq
  \frac{C |\ln \rho_N|}{\lambda^3} ,
\end{align}
where $C = C(\rho_n,|\r_N-\r_n|) > 0$, for all $\rho_N > 0$
sufficiently small, we deduce
\begin{align}
  4\pi \dbar \sum_{n=1}^{N-1}   \sqrt{\rho_n\rho_N}
  \cos\theta_n^* F_{vs}\left(\sqrt{\frac{\rho_N}{\rho_n}},\frac{
  |\r_{n} -\r_N| }{ \sqrt{\rho_n\rho_N} }\right) \geq -C \dbar
  \rho_N^2 |\ln \rho_N|.  
\end{align}
It follows that 
\begin{align}
  G(\{\rho_n, \r_n \})
  &\geq 4\pi \rho_N^2 \ln\left(
    \frac{e^{1+2\gamma}}{4} \rho_N^2
    \right)+\dbar \frac{\pi^3}{8} \rho_N  -C \dbar
    \rho_N^2 |\ln \rho_N| >0 
\end{align}
for any $\dbar>0$ and all $\rho_N>0$ small enough.  \vskip 0.2cm

\noindent {\bf Step 4.} It remains to show that $\rho_n^* < L_0^{-1}$
with $L_0 > \bar L_0$ for all $1 \leq n \leq N$. Assume this is not
true and, without loss of generality, that $\rho_N^* =L_0^{-1}$. In
this case it is not difficult to show that for all
$\bar \delta < \bar \delta_0$ with some $\dbar_0 > 0$ small enough
depending only on $N$ and $L_0$ we have
\begin{align}
  \inf_{\rho_n} F(\{ \rho_n, \r_n\}) > \lim_{l \to 0} F(\{ \rho_{n,l},
  \r_n\}), 
\end{align}
where $\rho_{n,l} = \rho^*_n$ for all $1 \leq n \leq N - 1$ and
$\rho_{N,l} \to 0$ as $l \to \infty$.  Therefore,
$\rho_{N} = L_0^{-1}$ cannot be a minimizer and
$\rho^*_n \in (0,L_0^{-1})$ for all $1 \leq n \leq N$. Finally, as the
function $F_N (\{\rho_n, \r_n\})$ is continuous for
$\rho_n \in (0, L_0^{-1})$, the minimum is attained in
$\mathcal A_N(\{\r_n\})$ for all $\bar \delta < \bar \delta_0$.
\end{proof}

Having established existence of minimizers of
$F_N(\{ \rho_n, \theta_n, \r_n\})$ over $\mathcal A_N(\{\r_n\})$, it
would next be interesting to understand the nature of the minimizers
of $F_N(\{\r_n\})$ with respect to the positions $\r_n$ of the
skyrmions in each layer. This, however, is in general a daunting task,
as the interaction energy describes a fully coupled, strongly
nonlinear system, in which the dependence on positions arises, in
addition to the direct interaction term, due to the implicit and a
priori unknown dependence of the skyrmion radii $\rho_n$ and angles
$\theta_n$ on the skyrmion positions $\{\r_n\}$. In particular, it is
not a priori clear whether minimizers of $F_N(\{\r_n\})$ over the
positions exist, as the interaction between different skyrmions could
be repulsive, leading to failure of compactness of the minimizing
sequences and, as a consequence, to failure of existence of
minimizers.

\section{Application to the case of bilayers in the absence of DMI}
\label{s:bilayer}

Instead of treating the problem in its full generality, in the
remainder of this paper we focus on the simplest particular case of
stray field-coupled ferromagnetic bilayers in the absence of DMI. Here
we can obtain a complete characterization of the energy minimizers. We
show that minimizers indeed exist, signifying an attractive
interaction of the skyrmions in the adjacent layers. We find that the
energy minimizers consist of pairs of concentric identical skyrmions
of {\em N\'eel type}, except that their in-plane magnetizations are
anti-parallel. Furthermore, these skyrmion pairs are chiral, as the
sense of the magnetization rotation in each of the layers is fixed by
the direction of the magnetization at infinity. This is in contrast
with the case of ferromagnetic monolayers, which are known to support
stray field-stabilized skyrmions of {\em Bloch type} that can have two
chiralities.

We now define the energy by setting $N = 2$ and $\bar\kappa = 0$ in
\eqref{energy02}:
\begin{equation}\label{energy022}
\begin{aligned}
  F_2(\{\rho_n,\theta_n,\r_n\}) &= \sum_{n=1}^2 \left[ - 4\pi \rho_n^2
    \ln\left( \frac{e^{1+2\gamma}}{4} \rho_n^2 \right) +
    \dbar \frac{\pi^3}{8} \rho_n (3\cos^2\theta_n  - 1) \right]  \\
  &+ \dbar \frac{\pi^3}{4} \sqrt{\rho_1 \rho_2} \left[ 3\cos\theta_1
    \cos\theta_2 F_{vv} \left( \sqrt{\frac{\rho_2}{\rho_1} },\frac{
        |\r_{2} -\r_1|}{ \sqrt{\rho_1\rho_2} } \right) - F_{ss} \left(
      \sqrt{\frac{\rho_2}{\rho_1} },\frac{ |\r_{2} -\r_1| }{
        \sqrt{\rho_1\rho_2} } \right) \right] \\
  &- 4\pi \dbar \sqrt{\rho_1\rho_2} \left[ \cos\theta_1
    F_{vs}\left(\sqrt{\frac{\rho_2}{\rho_1}},\frac{ |\r_{2} -\r_1| }{
        \sqrt{\rho_1 \rho_2} }\right) - \cos\theta_2
    F_{vs}\left(\sqrt{\frac{\rho_1}{\rho_2}},\frac{ |\r_{2} -\r_1| }{
        \sqrt{\rho_1\rho_2} }\right) \right].
\end{aligned}
\end{equation}
and would like to investigate the energy minimizing configurations of
stacked skyrmion pairs in a bilayer. We have the following result that
gives a complete characterization of the minimizers of this problem.

\begin{theorem}
  \label{t:bilayer}
  Let $L_0 > \bar L_0$. Then there exists $\bar\delta_0 > 0$ such that
  for every $\dbar < \dbar_0$ the minimizers of the energy
  $F_2(\{\rho_n, \theta_n, \r_n\})$ among
  $\rho_{1,2} \in (0, L_0^{-1})$, $\theta_{1,2} \in [-\pi, \pi)$ and
  $\r_{1,2} \in \R^2$ exist and satisfy
\begin{enumerate}[i)]
\item $\r_1=\r_2$;
\item $\theta_1=0$, $\theta_2=-\pi$;
\item  $\rho_1=\rho_2=\rho$, where
  \begin{align}
    \label{eq:rho2}
    \rho = { (16 + \pi^2) \dbar\over -64 W_{-1}\left( -{16 +
    \pi^2  \over 128} e^{1 + \gamma} \dbar\right)},
  \end{align}
  and $W_{-1}(t)$ is the Lambert $W$ function.
\end{enumerate}
\end{theorem}
\begin{proof}
  We first observe that by Lemmas \ref{l:Fvvss} and \ref{l:Fvs} in the
  Appendix the absolute values of the functions
  $F_{vv}(\alpha,\lambda)$, $F_{ss}(\alpha,\lambda)$ and
  $F_{vs}(\alpha,\lambda)$ are maximized at $\lambda=0$ for fixed
  $\alpha > 0$. Moreover, at $\lambda=0$ all these functions are
  positive. It follows, that
\begin{equation}
\begin{aligned}
  F_2(\{\rho_n,\theta_n,\r_n\}) &\geq \sum_{n=1}^2 \left[ - 4\pi
    \rho_n^2 \ln\left( \frac{e^{1+2\gamma}}{4} \rho_n^2 \right) +
    \dbar
    \frac{\pi^3}{8} \rho_n (3\cos^2\theta_n  - 1) \right]  \\
  &- \dbar \frac{\pi^3}{4} \sqrt{\rho_1 \rho_2} \left[ 3|\cos\theta_2|
    |\cos\theta_1| F_{vv} \left( \sqrt{\frac{\rho_2}{\rho_1} },0
    \right) + F_{ss} \left(
      \sqrt{\frac{\rho_2}{\rho_1} }, 0 \right) \right] \\
  &- 4\pi \dbar \sqrt{\rho_1\rho_2} \left[ |\cos\theta_1|
    F_{vs}\left(\sqrt{\frac{\rho_2}{\rho_1}},0\right) + |\cos\theta_2|
    F_{vs}\left(\sqrt{\frac{\rho_1}{\rho_2}}, 0 \right) \right],
\end{aligned}
\end{equation}
and equality is achieved if and only if $\r_1=\r_2$,
$\cos\theta_1\geq 0$ and $\cos\theta_2\leq 0$.

We now observe by completing the square that the function
\begin{align}
  h_0(\theta_1,\theta_2)
  &=\frac{3\pi^3}{8} \dbar \left( \sum_{n=1}^2
    \rho_n \cos^2\theta_n - 2  \sqrt{\rho_1
    \rho_2} F_{vv}\left(
    \sqrt{\frac{\rho_2}{\rho_1} }, 0\right)
    |\cos\theta_1| |\cos\theta_2|
    \right) \label{11} 
\end{align}
is non-negative due to the fact that
$F_{vv}\left( \alpha, 0\right) \leq 1$ with equality achieved only at
$\alpha = 1$ (see Lemma \ref{l:Fvvss} in the Appendix). Moreover, for
$\rho_n>0$ the function $h_0(\theta_1,\theta_2) =0$ if and only if one
of the two alternatives below holds:
\begin{enumerate}[i)]
\item $\rho_1 \not= \rho_2$ and $|\theta_{1,2}| = \frac{\pi}{2}$;
\item $\rho_1=\rho_2$ and $|\cos\theta_1|=|\cos\theta_2|$.
\end{enumerate}
We also observe that 
\begin{equation}
  -|\cos\theta_1| F_{vs}\left(\sqrt{\frac{\rho_2}{\rho_1}},0\right)  -
  |\cos\theta_2| F_{vs}\left(\sqrt{\frac{\rho_1}{\rho_2}},0\right) \geq
  -F_{vs}\left(\sqrt{\frac{\rho_2}{\rho_1}}, 0\right)
  -F_{vs}\left(\sqrt{\frac{\rho_1}{\rho_2}},0\right) , 
\end{equation}
with equality achieved at $\theta_{1,2}= 0$ or $\theta_{1,2}=
\pi$. Minimizing the right-hand side of the above expression in
$\rho_1,\rho_2$ we obtain $\rho_1=\rho_2$ and hence the following
lower bound achievable if $\rho_1=\rho_2$:
\begin{equation}
  -F_{vs}\left(\sqrt{\frac{\rho_2}{\rho_1}}, 0\right)
  -F_{vs}\left(\sqrt{\frac{\rho_1}{\rho_2}},0\right) \geq -2, 
\end{equation}
where we used Lemma \ref{l:Fvs} in the Appendix. Recalling that
$0<F_{ss}(\alpha,0)\leq 1$ with equality achievable if and only if
$\alpha=1$ (see Lemma \ref{l:Fvvss} of the Appendix) and combining our
findings above, we obtain
\begin{equation}
\begin{aligned}
  F_2(\{\rho_n,\theta_n,\r_n\}) &\geq \sum_{n=1}^2 \left[ - 4\pi
    \rho_n^2 \ln\left( \frac{e^{1+2\gamma}}{4} \rho_n^2 \right) -
    \dbar \frac{\pi^3}{8} \rho_n \right] - \dbar \frac{\pi^3}{4}
  \sqrt{\rho_1 \rho_2} -  8\pi \dbar \sqrt{\rho_1\rho_2}  \\
  &\geq \sum_{n=1}^2 \left[ - 4\pi \rho_n^2 \ln\left(
      \frac{e^{1+2\gamma}}{4} \rho_n^2 \right) - \dbar
    \frac{\pi^3}{4} \rho_n - 4\pi \dbar \rho_n\right] \geq 2 F(\rho), 
\end{aligned}
\end{equation}
with equality achieved if and only if $\r_1=\r_2$, $\theta_1=0$ and
$\theta_2=\pi$, and $\rho_1=\rho_2=\rho$, where
\begin{align}
F(\rho) = - 4\pi\rho^2 \ln\left( \frac{e^{1+2\gamma}}{4}\rho^2 \right)
  - \dbar \frac{\pi^3}{4} {\rho} -4\pi \dbar \rho. 
\end{align}
The graph of $F(\rho)$ is illustrated in Fig. \ref{f:Frho}.

\begin{figure}
  \centering
  \includegraphics[width=8cm]{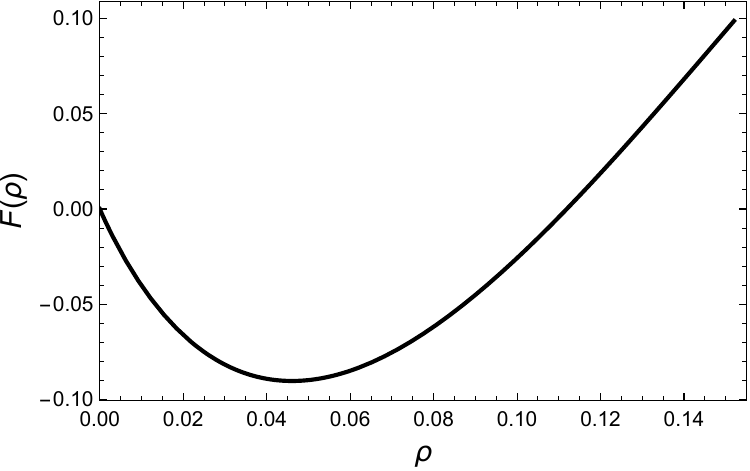}
  \caption{Plot of $F(\rho)$ when $\rho \in (0, \rho_0)$ with
    $\rho_0 = \bar L_0^{-1}$ for $\bar \delta = 0.25$. }
  \label{f:Frho}
\end{figure}

It is now clear that
\begin{align}
  \inf_{\rho_n} \min_{\theta_n, \r_n} F_2(\{ \rho_n,
  \theta_n, \r_n\}) = 2 \inf_\rho F(\rho).
\end{align}
We can minimize $F(\rho)$ in the admissible interval of
$\rho \in (0, L_0^{-1})$. It is easy to check that for
$L_0 > \bar L_0$ and $\dbar_0 > 0$ sufficiently small depending only
on $L_0$ the minimum of $F(\rho)$ is attained at the point
$\rho \in (0, L_0^{-1})$ satisfying
\begin{align}
  0 = F'(\rho)=-8 \pi \rho\left[ \ln(\rho^2
  \frac{e^{1+2\gamma}}{4})+1\right]- \frac{\dbar \pi^3}{4}-4\pi \dbar.  
\end{align}
After a little algebra, one can see that the solution of this equation
in the interval $(0, \rho_0)$, where $\rho_0 = \bar L_0^{-1}$, is
given by \eqref{eq:rho2}. This concludes the proof.
\end{proof}

To summarize, we have demonstrated that the global energy minimizers
of $F_2(\{\rho_n, \theta_n, \r_n\})$ are characterized by a certain
symmetry that makes the in-plane components of the magnetization in
\eqref{eq:prof} anti-parallel. It is instructive to see what this
assumption would lead to on the level of the original energy in
\eqref{eq:EN}, before introducing the truncated BP ansatz. Setting
$(\m_1^\perp, m_1^\|) = (-\m_2^\perp, m_2^\|) = \m$, we see that the
energy $\bar E_2$ of the configuration $\{\m_1, \m_2\}$ becomes
$\bar E_2(\{\m_1, \m_2\}) = 2 \bar E^\pm(\m)$, where
\begin{align}
  \bar E^\pm(\m) = \int_{\R^2} \left( |\nabla
  \m|^2 + |\m^\perp|^2 - \dbar \m^\perp \cdot \nabla m^\| \right)
  d^2 r  - 2 \dbar\int_{\R^2} \int_{\R^2} {(m^\|(\r) -
  m^\|(\r'))^2 \over 8 \pi  |\r - \r'|^3} \, d^2 r \, d^2 r'.
\end{align}
One can observe that for these configurations the volume-surface
interactions act as an effective interfacial DMI term, favoring the
N\'eel rotation of the magnetization with a particular rotation
sense. At the same time, the volume-surface and surface-surface
interactions act constructively to stabilize the skyrmion pair, while
the penalizing volume-volume charge interaction is absent. The reason
for the stabilizing action of the volume-surface interaction may be
seen from Fig. \ref{f:charges}. For skyrmions with anti-parallel
in-plane magnetizations in the two layers and counter-clockwise
rotation in the bottom layer the volume-surface interaction energy is
lower than that of all other possible skyrmion configurations with
anti-parallel in-plane magnetizations and the same out-of-plane
component. Notice that the stray field energy due to the volume-volume
interactions may be forced to be zero by choosing the Bloch rotation
in both layers instead (not necessarily anti-parallel). However, in
that case the DMI-like term due to the volume-surface interaction does
not contribute to the energy of the skyrmion pair, making the Bloch
rotation as compared to the N\'eel rotation with anti-parallel
in-plane components of the magnetization less favorable. Finally,
notice that the configuration in Fig. \ref{f:charges}(f) that
corresponds to the lowest energy is reminiscent of a flux closure
structure in bulk ferromagnets.

\begin{figure}
  \centering
  \includegraphics[width=15cm]{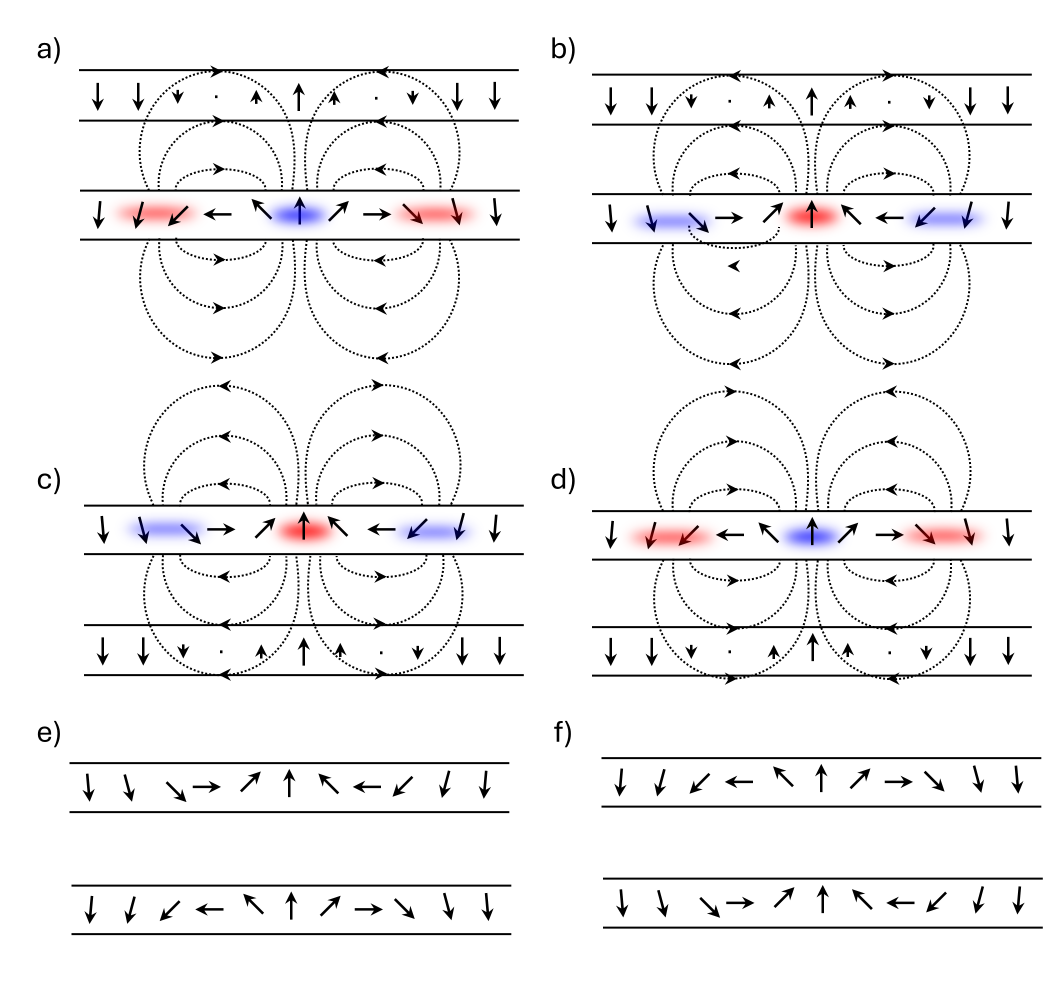}
  \caption{An illustration of the stray field interactions between the
    volume and surface charges in a bilayer for a skyrmion with
    anti-parallel in-plane magnetization components (i.e., with
    $m_1^\| = m_2^\|$ and $\m_1^\perp = -\m_2^\perp$): (a,c,e)
    clockwise rotation in the bottom layer; (b,d,f) anti-clockwise
    rotation in the bottom layer. The volume charge density
    $\uprho_\m^\mathrm{vol} = -\nabla_\perp \cdot \m^\perp$ in one
    layer is indicated in blue (negative) and red (positive), and its
    associated magnetic field lines are shown by the lines with arrows
    going from red to blue regions. Only the out-of-plane component
    $m_n^\|$ of the magnetization in the other layer that contributes
    to the volume-surface interaction is shown in (a-d), while the
    corresponding full magnetization profiles are shown in (e,f).  In
    (a,c), the out-of-plane magnetic moments of one layer point agains
    the field lines from the other layer, resulting in a higher
    energy.  In (b,d), the out-of-plane magnetic moments of one layer
    point along the field lines from the other layer, resulting in a
    lower energy.  The total surface-surface interaction energies are
    identical in both cases, and the total volume-volume interaction
    energy is asymptotically zero.}
  \label{f:charges}
\end{figure}

To verify the predictions of our asymptotic analysis, we carried out a
numerical study of skyrmion profiles in stray field-coupled
ferromagnetic bilayers, using {\sc MuMax3} software
\cite{leliaert18}. For the material parameters, we chose those
corresponding to a ferrimagnetic material such as GdCo with the
material parameters $A = 20$ pJ/m, $M_s = 10^5 $A/m, $K_u = 6700$
J/m$^3$ \cite{caretta18}. We consider two 5 nm-thick layers of such a
material with a non-magnetic separator of negligible thickness. These
parameters give an exchange length $\ell_{ex} \simeq 56.4$ nm and a
small dimensionless layer thickness $\delta \simeq 0.089$, justifying
the use of the reduced model appropriate for ultrathin films. The
quality factor associated with the uniaxial magnetocrystalline
anisotropy is $Q \simeq 1.066$, giving the parameter
$\bar \delta \simeq 0.344$ characterizing the strength of the stray
field interaction that is also within the validity range of the
asymptotic theory in $\bar \delta \ll 1$. For this value of
$\bar \delta$, the formula in \eqref{eq:rho2} predicts
$\rho \simeq 0.0938$, which corresponds to the dimensional skyrmion
radius of 20.5 nm. The resulting profiles obtained, using the {\tt
  minimize} function of {\sc MuMax3} on a $2048 \times 2048 \times 2$
grid with the in-plane discretization steps
$\Delta x = \Delta y = 0.5$ nm, the out-of-plane step equal to $d = 5$
nm and the number of repeats in $(X, Y, Z)$ set to $(5, 5, 0)$ to
approximate the periodic boundary conditions, are presented in
Fig. \ref{f:bs}. In a very good agreement with the theoretical
prediction, the obtained profiles consist of a pair of concentric
N\'eel skyrmions that are very close to the Belavin-Polyakov profiles
with radius 20 nm obtained by fitting the numerical profiles to the
Belavin-Polyakov profiles.

\begin{figure}
  \centering  \includegraphics[width=16cm]{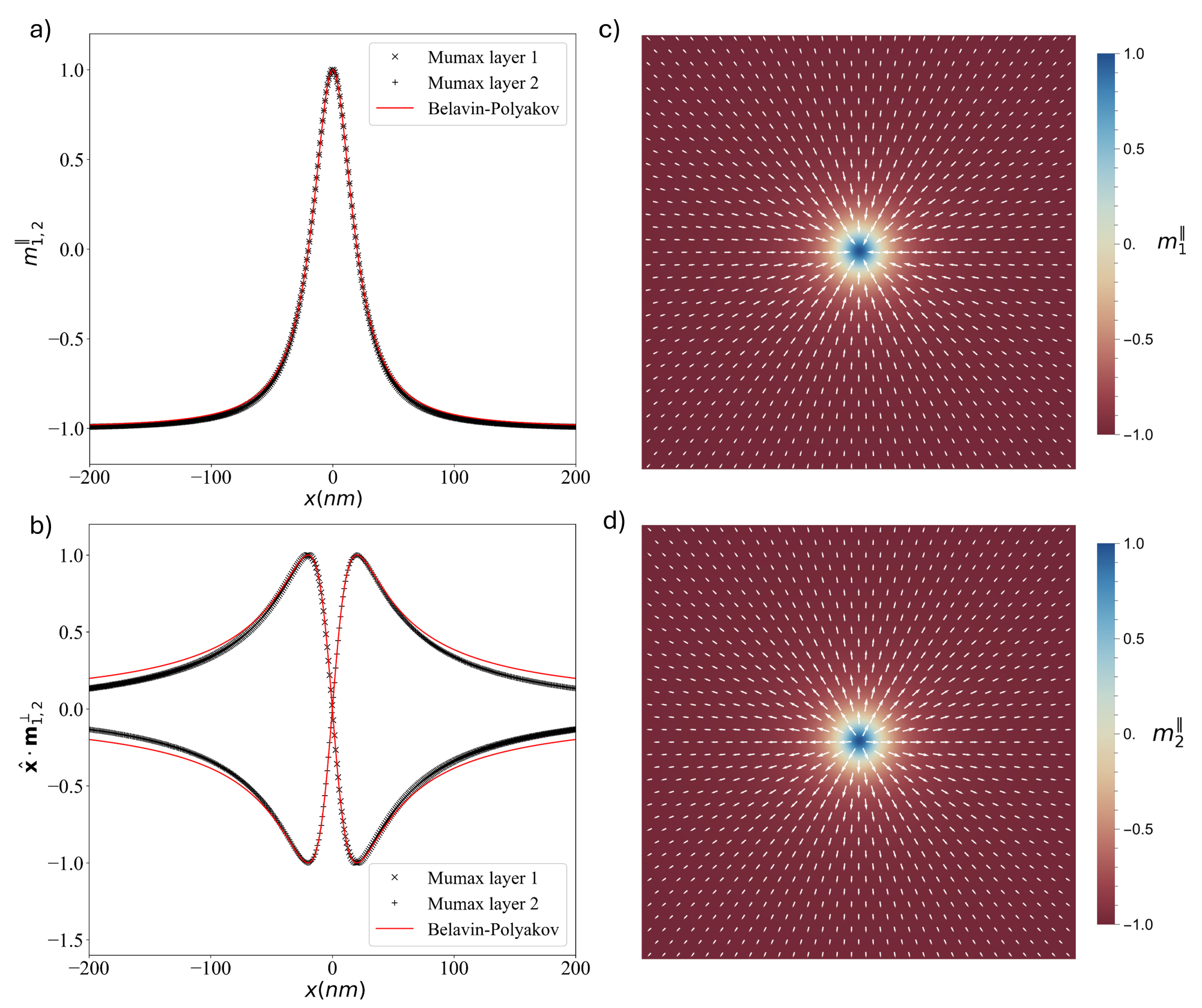}
  \caption{Stray field-stabilized N\'eel skyrmion in an ultrathin
    ferromagnetic bilayer. (a) The out-of-plane components
    $m^\|_{1,2}$ of the magnetization in the bottom and top layers,
    respectively, along the horizontal line through the skyrmion
    center. (b) The in-plane components
    $\hat{\mathbf{x}} \cdot \m^\perp_{1,2}$ in the bottom and top
    layers, respectively, along the horizontal line through the
    skyrmion center. (c) and (d) The top view of the magnetization in
    the bottom and top layers, respectively. Results of the {\sc
      MuMax3} simulations on the $2048 \times 2048 \times 2$ grid with
    the in-plane discretization steps $\Delta x = \Delta y = 0.5$ nm,
    the out-of-plane discretization step $\Delta z$ equal to the
    single layer thickness $d = 5$ nm, and periodic boundary
    conditions in the plane.  The material parameters are $A = 20$
    pJ/m, $M_s = 10^5 $A/m, $K_u = 6700$ J/m$^3$, corresponding to a
    ferrimagnetic material, with a non-magnetic spacer of negligible
    thickness. In all the panels, only a 400 nm $\times$ 400 nm region
    around the skyrmion center is shown.  In (a) and (b), the N\'eel
    Belavin-Polyakov profiles with radius 20 nm are also shown by thin
    red lines.}
  \label{f:bs}
\end{figure}

We conclude our study by attempting to calculate the function
$F_2(\{\r_1, \r_2\})$ from Theorem \ref{t:FN} that determines the
interaction energy of two skyrmions in the adjacent layers of a
bilayer. By the translational and rotational symmetries of the problem
this function depends only on $r = |\r_1 - \r_2|$ and is obtained by
minimizing the function $F_2(\{ \rho_n, \theta_n, \r_n\})$ in four
parameters $\rho_{1,2}$ and $\theta_{1,2}$. Nevertheless, the problem
is still intractable analytically, since no manageable closed form
analytical expressions for the functions $F_{vv}(\alpha, \lambda)$,
$F_{ss}(\alpha, \lambda)$ and $F_{vs}(\alpha, \lambda)$ are available
for $\lambda > 0$ and general $\alpha > 0$. It would seem plausible,
and is supported by the numerical evaluation of the energy at a few
points of the parameter space that the minimizers from Theorem
\ref{t:FN} in the case of bilayers, $N = 2$, would exhibit a symmetry
for the skyrmion radius with respect to the layer position, which is
also certainly true for the global energy minimizers by Theorem
\ref{t:bilayer}. Thus, it is reasonable to conjecture that in a
minimizer from Theorem \ref{t:FN} at fixed $r > 0$ we have
$\rho_1 = \rho_2 = \rho$, and thus
$F_2(\{\r_1, \r_2\}) = F_2^\mathrm{sym}(r)$, where
$F_2^\mathrm{sym}(r)$ can now be calculated in closed form. Indeed, we
have
\begin{align}
  F_{vv}(1, \lambda) = \frac{32}{3 \pi ^2 \lambda ^2 \left(\lambda 
  ^2+4\right)^{3/2}} \Bigg[ \left(\lambda ^2 \sqrt{\lambda ^2+4} +4
  \sqrt{\lambda ^2+4}+8\right) K^2\left(\frac{1}{2}-\frac{\sqrt{\lambda
  ^2+4}}{4}\right)\notag \\
  -4 \left( \lambda ^2 \sqrt{\lambda ^2+4}+2
  \sqrt{\lambda ^2+4}+4\right) K\left(\frac{1}{2}-\frac{\sqrt{\lambda
  ^2+4}}{4}\right) E\left(\frac{1}{2}-\frac{\sqrt{\lambda
  ^2+4}}{4}\right) \\
  +4 \left(\lambda ^2+2\right) \sqrt{\lambda ^2+4} \
  E^2\left(\frac{1}{2}-\frac{\sqrt{\lambda 
  ^2+4}}{4}\right)\Bigg], \notag
\end{align}
where, again, $K(m)$ and $E(m)$ are the complete elliptic integrals of
the first and second kind, respectively. Similarly
\begin{align}
  F_{ss}(1, \lambda) = -\frac{32}{\pi ^2 \lambda ^2 \left(\lambda 
  ^2+4\right)^{3/2}} \Bigg[ \left(\left(\sqrt{\lambda ^2+4}+4\right)
  \lambda ^2+4 \left(\sqrt{\lambda ^2+4}+2\right)\right) 
  K^2\left(\frac{1}{2}-\frac{\sqrt{\lambda ^2+4}}{4}\right)\notag \\
  -8
  \left(\lambda ^2+\sqrt{\lambda ^2+4}+2\right)
  K\left(\frac{1}{2}-\frac{\sqrt{\lambda 
  ^2+4}}{4}\right) E\left(\frac{1}{2}-\frac{\sqrt{\lambda
  ^2+4}}{4}\right) \\
  +8 \sqrt{\lambda ^2+4} \
  E^2\left(\frac{1}{2}-\frac{\sqrt{\lambda 
  ^2+4}}{4}\right)\Biggr], \notag
\end{align}
and
\begin{align}
  F_{vs}(1, \lambda) = 2 \left(\frac{1}{\lambda ^2+4}+\frac{4 \sinh
  ^{-1}\left(\frac{\lambda }{2}\right)}{\lambda  \left(\lambda
  ^2+4\right)^{3/2}}\right).   
\end{align}
With these expressions for $F_{vv}(\alpha, \lambda)$,
$F_{ss}(\alpha, \lambda)$ and $F_{vs}(\alpha, \lambda)$, we can
proceed to minimize the energy in \eqref{energy022} with respect to
the angles $\theta_{1,2}$ to obtain a closed form expression for
$F_2^\mathrm{sym} (\rho, r) = \displaystyle \min_{\theta_1, \theta_2}
F_2 (\rho_1, \theta_1, \rho_2, \theta_2, r)$ with
$\rho_1 = \rho_2 = \rho$. The plot of this function for a particular
value of $\bar \delta$ is presented in Fig. \ref{f:F2sym}.

\begin{figure}
  \centering
  \includegraphics[width=12cm]{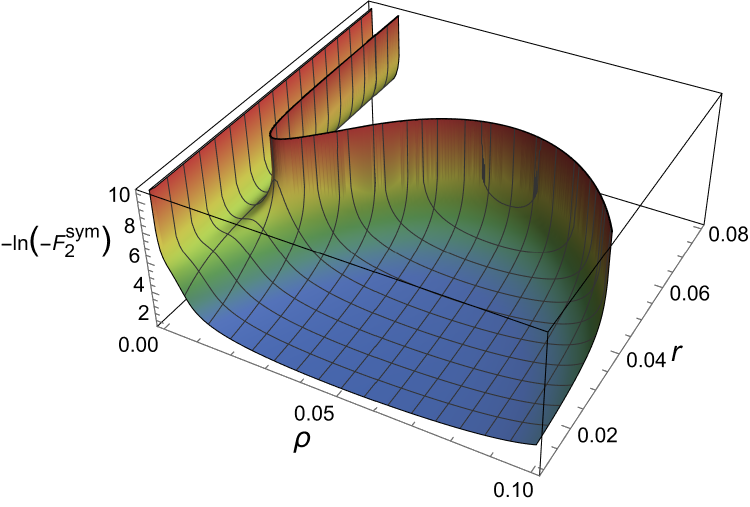}
  \caption{Plot of $- \ln \left[ -F_2^\mathrm{sym}(\rho, r) \right]$
    for $\dbar = 0.25$. The choice of the function plotted helps to
    visualize the part of the parameter space in which the energy is
    negative. }
  \label{f:F2sym}
\end{figure}

The function $F_2^\mathrm{sym} (\rho, r)$ can finally be minimized
numerically for a given value of $r > 0$. The results for a particular
choice $\dbar = 0.25$ are presented in Fig. \ref{f:F2symmin}, where
the minimal energy, the optimal value of $\rho$ and the optimal value
of $\cos \theta_1 = -\cos \theta_2 > 0$ (the latter is due to the fact
that from $0 < F_{vv}(1, \lambda) < 1$ for $\lambda > 0$ follows that
the energy is a strictly convex function of $\cos \theta_1$ and
$\cos \theta_2$) are plotted. We observe that the minimum of the
energy is indeed attained at $r = 0$, as it should be by Theorem
\ref{t:bilayer}. As the value of $r$ increases, the equilibrium radius
$\rho$ also slightly increases, while the optimal angles remain those
of the minimizer: $\theta_1 = 0$ and $\theta_2 = -\pi$, in Theorem
\ref{t:bilayer}. We note that as can be seen from Fig. \ref{f:F2sym},
the energy landscape in the $(\rho, r)$ plane consists of a wide
valley for not too big values of $\rho \sim r$ and a narrow gorge
providing an escape path to $r \to \infty$ and corresponding to an
almost constant value of $\rho$ and increasing values of $r$. At
$r \gg 1$ the latter corresponds to a pair of non-interacting
skyrmions in each layer.

When the value of $r$ is increased from zero, at first the minimum is
found in the wide valley, but as the parameters push the energy
towards the steep wall at a certain critical value of $r$ (close to
0.054 for $\dbar = 0.25$) the minimizer jumps into the gorge and
continues to follow along it. Notice that for this value of $\dbar$
the angles switch abruptly from those corresponding to the N\'eel
skyrmions in a global energy minimizer to
$|\theta_{1,2}| \simeq {\pi \over 2}$ corresponding to a pair of Bloch
skyrmions. Also notice that the radius of the skyrmions in the gorge
is an order of magnitude smaller than that of the global minimizer,
and the absolute value of its energy is two orders of magnitude lower,
respectively. In sum, the skyrmions remain of N\'eel type and exhibit
a strongly attractive interaction within a certain range of separation
distances. Beyond that range they abruptly change their nature and
behave as a pair of weakly interacting Bloch skyrmions, and at large
separations these skyrmions exhibit a weak repulsive interaction
dominated by the dipolar forces. The strongly nonlinear attractive
interaction of a pair of skyrmions at sufficiently short distances
thus may provide an important stabilization mechanism that is highly
desirable for spintronic applications.

\begin{figure}
  \centering
  \includegraphics[width=16cm]{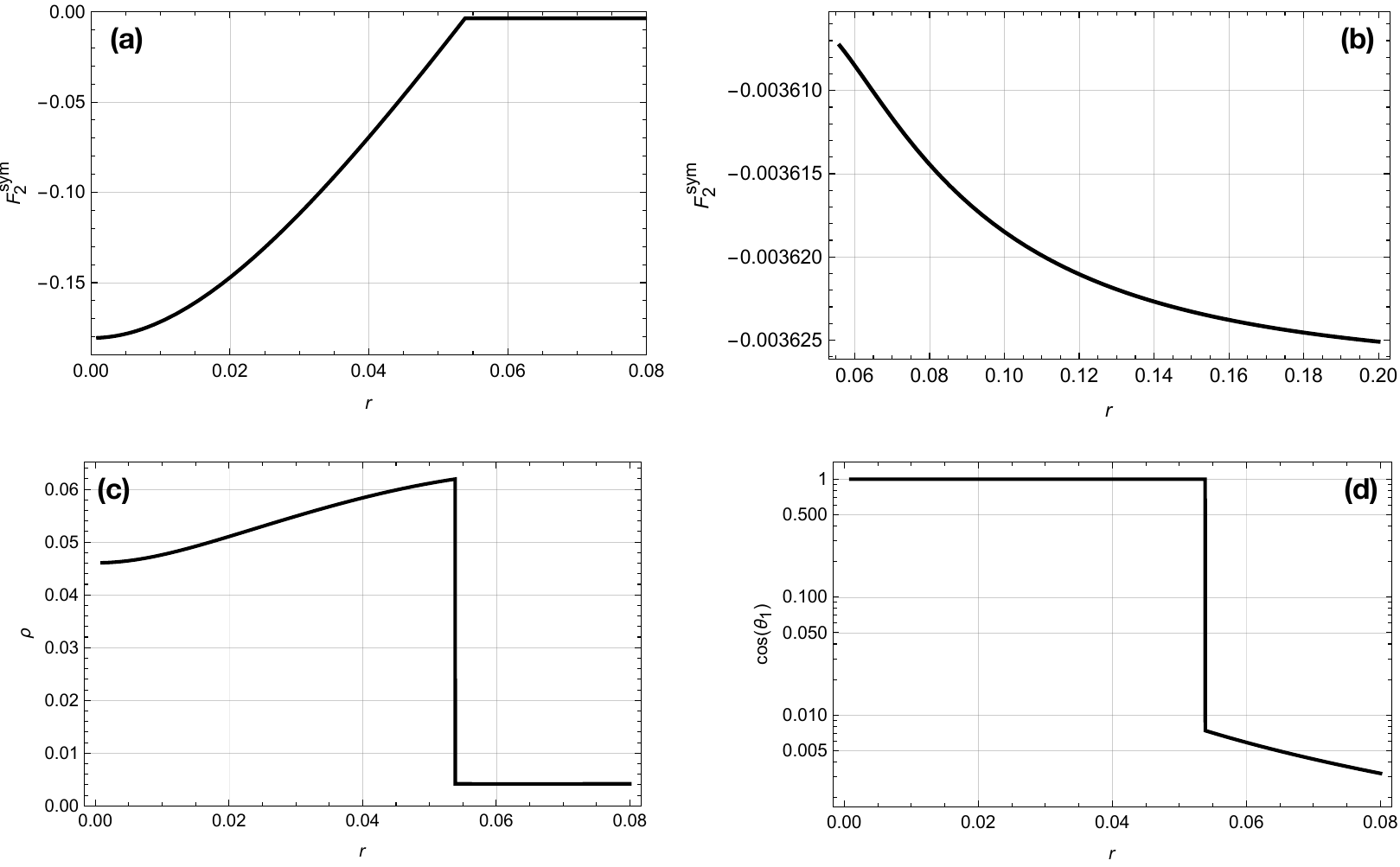}
  \caption{The result of the global numerical minimization of
    $F_2^\mathrm{sym}(\rho, r)$ in $\rho$ for $\dbar = 0.25$: (a) the
    minimum energy $F_2^\mathrm{sym}(r)$; (b) zooming in on the
    minimum energy $F_2^\mathrm{sym}(r)$ to show the repulsion at
    large separation distances; (c) optimal value of $\rho$ as a
    function of $r$; (d) optimal value of $\cos \theta_1$ as a
    function of $r$.}
  \label{f:F2symmin}
\end{figure}


\appendix 
\section{Appendix}
\label{s:appendix}

In this section, we establish a few basic properties of the functions
$F_{vv}(\alpha, \lambda)$, $F_{ss}(\alpha, \lambda)$ and
$F_{vs}(\alpha, \lambda)$ necessary in proving our theorems. An
impatient reader may quickly convince oneself about these facts by
simply plotting those functions evaluated numerically from their
respective integral definitions.

\begin{lemma}
  \label{l:Fvvss}
  For $\alpha \in (0, 1)$, let $F_{vv}(\alpha, 0)$ and
  $F_{ss}(\alpha, 0)$ be defined by \eqref{eq:Fvv0} and
  \eqref{eq:Fss0}, respectively. Then $0 < F_{vv}(\alpha, 0) < 1$ and
  $0 < F_{ss}(\alpha, 0) < 1$. Furthermore, for all $\alpha > 0$ and
  $\lambda \geq 0$
  \begin{align}
    |F_{vv}(\alpha, \lambda)| \leq F_{vv}(\alpha, 0), \qquad
    |F_{ss}(\alpha, \lambda)| \leq F_{ss}(\alpha, 0),
  \end{align}
  where $F_{vv}(\alpha, \lambda)$ and $F_{ss}(\alpha, \lambda)$ are
  defined in \eqref{eq:Fvvdef} and \eqref{eq:Fssdef}, and $F_{vv}(\alpha, \lambda)$ and
  $F_{ss}(\alpha, \lambda)$ are uniquely maximized by
  $(\alpha, \lambda) = (1, 0)$, with $\max F_{vv} = \max F_{ss} = 1$.
\end{lemma}

\begin{proof}
  The statements about $F_{vv}(\alpha, \lambda)$ and
  $F_{ss}(\alpha, \lambda)$ follow immediately from the first part of
  the lemma, since from the definitions of $F_{vv}(\alpha, \lambda)$
  and $F_{ss}(\alpha, \lambda)$, and from the fact that
  $|J_0(\lambda \xi)| < 1$ for all $\lambda > 0$ and $\xi > 0$ we have
  $|F_{vv}(\alpha, \lambda)| < F_{vv}(\alpha, 0)$ and
  $|F_{ss}(\alpha, \lambda)| < F_{ss}(\alpha, 0)$ for all $\alpha > 0$
  and $\lambda > 0$, together with the facts that
  $F_{vv}(\alpha, 0) = F_{vv}(\alpha^{-1}, 0)$ and
  $F_{ss}(\alpha, 0) = F_{ss}(\alpha^{-1}, 0)$.

  We now prove that $F_{ss}(\alpha, 0)$ attains its maximum only at
  $\alpha = 1$. Since $F_{ss}(\alpha, 0) = F_{ss}(\alpha^{-1}, 0)$ and
  $F_{ss}(1,0) = 1$, we only need to investigate the interval of
  $\alpha \in (0, 1)$. Recalling the definition of $F_{ss}(\alpha,0)$
  and using the change of variables $t=1-\alpha^4$ with $t \in (0,1)$,
  we note that
\begin{align} 
  F_{ss}((1-t)^\frac14, 0) = \frac{16 (1-t)^\frac34}{\pi t^2} \left(
  (2-t)K(t)-2E(t) \right). 
\end{align}
We now need to show that $f(t)=F_{ss}((1-t)^\frac14, 0)$ has a unique
maximum at $t=0$. We can differentiate to obtain
\begin{align} 
  f'(t)= -\frac{4}{\pi t^3 (1-t)^\frac14} \left( (3t^2 -16 t+16) K(t)
  +8(t-2) E(t) \right). 
\end{align}
If we show that $f_1(t) = (3t^2 -16 t+16) K(t) +8(t-2) E(t) >0$ on
$(0,1)$ then it follows that $f'(t)<0$ on $(0,1)$, and hence the
maximum of $F_{ss}(\alpha, 0)$ is achieved at $\alpha = 1$.

We note that $f_1(0)=0$ and 
\begin{align} 
 f_1'(t)= -\frac{3}{2 (1-t)} \left( (3t^2 -11 t+8) K(t) +(7t-8) E(t) \right).
\end{align}
Hence we need to show that
$f_2(t)= (3t^2 -11 t+8) K(t) +(7t-8) E(t)<0$. Again, we note that
$f_2(0)=0$ and
\begin{align} 
 f_2'(t)= \frac92(t-2) K(t) +9 E(t) .
\end{align}
Hence we need to show that $f_3(t) = \frac92(t-2) K(t) +9 E(t) <0$. We
note that $f_3(0)=0$ and
\begin{align} 
 f_3'(t)= -\frac{9}{4(1-t)}( K(t) (t-1) +E(t)) .
\end{align}
Finally, we need to show that $f_4(t)= K(t) (t-1) +E(t)>0$. We note
that $f_4(0)=0$ and
\begin{align} 
 f_4'(t)= \frac12 K(t) >0.
\end{align}
The result is proved.

We now prove that $F_{vv}(\alpha, 0)$ attains its maximum only at
$\alpha = 1$. Similarly, we consider only the interval of
$\alpha \in (0,1)$. Recalling the definition of $F_{vv}(\alpha,0)$ and
using the change of variables $t=1-\alpha^4$ with $t \in (0,1)$, we
obtain
\begin{align} 
  F_{vv}((1-t)^\frac14, 0) = \frac{16 (1-t)^\frac14}{3\pi t^2} \left(
  (2-t)E(t)-2(1-t)K(t) \right). 
\end{align}

We now need to show that $g(t)=F_{vv}((1-t)^\frac14, 0)$ has maximum
at $t=0$. We can differentiate to obtain
\begin{align} 
  g'(t)= \frac{4}{3\pi t^3 (1-t)^\frac34} \left(  8(t^2 -3 t+2) K(t)
  -(t^2-16t+16) E(t) \right). 
\end{align}
If we show that $g_1(t) = 8(t^2 -3 t+2) K(t) -(t^2-16t+16) E(t) <0$ on
$(0,1)$ then it follows that $g'(t)<0$ on $(0,1)$, and, hence, the
maximum of $F_{vv}(\alpha, 0)$ is achieved at $\alpha=1$.

Arguing as in the case of $f(t)$, we note that $g_1(0)=0$ and 
\begin{align} 
 g_1'(t)= -\frac52 \left( (8-5t) K(t) +(t-8) E(t) \right).
\end{align}
We now need to show that
$g_2(t)=-\left( (8-5t) K(t) +(t-8) E(t) \right)<0$. Again, we note
that $g_2(0)=0$ and
\begin{align} 
  g_2'(t)= -\frac{3}{2(t-1)} \left(2(1 - t) K(t) +(t-2)E(t) \right) .
\end{align}
Hence, we now need to show that
$g_3(t) = -\left( 2(1-t) K(t) + (t-2)E(t) \right) >0$. We note that
$g_3(0)=0$ and
\begin{align} 
 g_3'(t)= \frac32( K(t) - E(t)) >0.
\end{align}
The result is proved.
\end{proof}

\begin{lemma}
  \label{l:Fvs}
  For $\alpha > 0$, let $F_{vs}(\alpha, 0)$ be defined by
  \eqref{eq:Fvs0}. Then $0 < F_{vs}(\alpha, 0) \leq 2$, and for all
  $\lambda \geq 0$
  \begin{align}
    |F_{vs}(\alpha, \lambda)| \leq F_{vs}(\alpha, 0),
  \end{align}
  where the function $F_{vs}(\alpha, \lambda)$ is defined in
  \eqref{eq:Fvsdef}. Furthermore,
  \begin{align}
    \label{eq:Fvs2}
    |F_{vs}(\alpha, \lambda)| + |F_{vs}(\alpha^{-1}, \lambda)| \leq 2, 
  \end{align}
  with equality achieved only at $(\alpha, \lambda) = (1, 0)$.
\end{lemma}

\begin{proof}
  As in the case with $F_{vv}(\alpha, \lambda)$ and
  $F_{ss}(\alpha, \lambda)$, the statements about
  $F_{vs}(\alpha, \lambda)$ follow, once we demonstrate the desired
  properties of $F_{vs}(\alpha, 0)$, since by the same argument
  $|F_{vs}(\alpha, \lambda)| \leq F_{vs}(\alpha, 0)$. From the
  definition of the latter, it is clear that the function
  $F_{vs}(\alpha, 0)$ is positive and continuous for all $\alpha > 0$,
  including at $\alpha = 1$, since
  $\lim_{\alpha \to 1} F_{vs}(\alpha, 0) = 1$. Therefore, since
  $F_{vs} (\alpha, 0) \to 0$ as $\alpha \to 0$ or $\alpha \to \infty$,
  the function $F_{vs}(\alpha, 0)$ is uniformly bounded.

  To show the inequality in \eqref{eq:Fvs2}, we observe that
  \begin{align}
    F_{vs}(\alpha, 0) + F_{vs}(\alpha^{-1}, 0)
    & = \frac{2 \alpha
      \left(\alpha^2-1\right) \left(\alpha^4+4 \alpha^2 \ln
      \alpha -1\right)}{\left(\alpha^4-1\right)^2} \notag \\
    & = \sech \omega (1 + \omega \sech \omega \csch \omega), 
  \end{align}
  where $\omega = \ln \alpha $, which is manifestly maximized at
  $\omega = 0$ among all $\omega \in \R$. Finally, the upper bound on
  $F_{vs}(\alpha, 0)$ follows directly from \eqref{eq:Fvs2}.
\end{proof}

\bibliographystyle{plain}
\bibliography{multilayers}

\addcontentsline{toc}{section}{References}

\end{document}